\pdfoutput=1

\documentclass[english,11pt]{article}
\include{macros}

\begin{document}

\title{MINTS: Minimalist Thompson Sampling}

\author{Kaizheng Wang\thanks{Department of IEOR and Data Science Institute, Columbia University. Email: \texttt{kaizheng.wang@columbia.edu}.}
}

\date{This version: \today}

\maketitle

\begin{abstract}
The Bayesian paradigm offers principled tools for sequential decision-making under uncertainty, but its reliance on a probabilistic model for all parameters can hinder the incorporation of complex structural constraints. We introduce a minimalist Bayesian framework that places a prior only on the location of the optimum, while eliminating nuisance parameters through profile likelihood. This yields a generalized posterior that naturally accommodates structural constraints. As a direct instantiation, we develop MINimalist Thompson Sampling (MINTS). For multi-armed bandits with mean constraints, we establish near-optimal non-asymptotic regret guarantees and sharp almost-sure asymptotic regret characterizations. In particular, MINTS attains the classical Lai--Robbins constant in the unstructured setting and automatically adapts to unimodal structure, achieving the sharp constant determined only by the immediate neighbors of the optimal arm.
\end{abstract}

\section{Introduction}\label{sec-intro}

Effective sequential decision-making requires balancing exploration and exploitation: gathering information versus acting on current knowledge. Simple strategies such as explore-then-commit split the horizon into separate phases, but choosing the split is difficult. Adaptive frequentist methods guide exploration by adding uncertainty adjustments to certainty-equivalent decisions \citep{LRo85,ACF02,GLa97}, or by designing sampling rules that track the minimal exploration rates identified by regret lower bounds \citep{CMP17,DSK20,VGo24}. The Bayesian paradigm offers a principled alternative by maintaining a posterior distribution over unknown quantities \citep{Tho33,JSW98}. However, specifying a full prior can be difficult when prior information is expressed through structural constraints, such as shape restrictions, rather than through a tractable probabilistic model.

We develop a minimalist Bayesian framework that places a prior only on the location of the optimum and handles the remaining parameters through profile likelihood \citep{BCo94}. This reduced-dimensional prior, combined with profile likelihood, yields a generalized posterior distribution over the optimizer. Based on this, we propose a MINimalist Thompson Sampling (\Alg) algorithm, and derive sharp regret bounds for structured multi-armed bandits.

A preliminary version of this paper appeared in the NeurIPS 2025 MLxOR Workshop 
\citep{Wan25anonymous}. The present paper adds a sharp asymptotic analysis and numerical experiments.

\paragraph{Related work.}

Our work contributes a general Bayesian approach to stochastic optimization under structural constraints, a setting that has largely been studied case by case \citep{PTR17,SGF20}. The framework builds on Thompson sampling \citep{Tho33,GMM14,RVK18}, also known as posterior sampling or probability matching, because actions are drawn according to their posterior probability of being optimal \citep{Sco10}. Standard Thompson sampling induces this posterior from a probabilistic model for the full problem instance, whereas our approach works directly with the optimizer itself. Related ideas appear in full-information online learning \citep{LWa89,Vov90,CFH97,Cat04,BHW16}, but we study the more challenging partial-feedback setting.

Existing Bayesian methods that reason about the optimum are either tailored to specific structures \citep{WFH13}, or still rely on full probabilistic models for belief updates \citep{VVW09,HSc12,HHG14,RVa18}. Closest to our work, \citet{SNO21} augment a standard model with a prior on the optimum, whereas we use profile likelihood to avoid modeling nuisance parameters.



\section{Preliminaries}\label{sec-preliminaries}

A stochastic optimization agent seeks to maximize an unknown function $f$ over a decision set $\spaceofaction$:
\begin{align}
	\max_{\action \in \spaceofaction} f(\action).
	\label{eqn-problem}
\end{align}
The agent learns about the optimum by sequentially interacting with the environment. Starting from an empty dataset $\dataset_0=\varnothing$, at each period $t\in\ZZ_+$, the agent selects a decision $\action_t\in\spaceofaction$ based on past data $\dataset_{t-1}$, receives randomized feedback $\feedback_t$, and updates the dataset to $\dataset_t=\dataset_{t-1}\cup\{(\action_t,\feedback_t)\}$. Performance over $T$ periods is commonly measured by the cumulative regret $\sum_{t=1}^T  [\max_{\action\in\spaceofaction} f(\action)-f(\action_t) ]$ or the simple regret $\max_{\action\in\spaceofaction} f(\action)-f(\action_T)$.

\begin{example}[Multi-armed bandit with mean constraints]\label{example-MAB}
	The decision set is a collection of $K\in\ZZ_+$ arms, i.e., $\spaceofaction=[K]$. Each arm $\action$ is associated with a reward distribution $\distP_{\action}$ over $\RR$, and the objective value $f(\action)$ is the expected reward $\EE_{\rewardrv\sim\distP_{\action}}\rewardrv$. The mean vector $(f(1),\dots,f(K))$ belongs to a given set $\spaceofparameter\subseteq\RR^K$. At time $t$, the feedback $\feedback_t$ is drawn from $\distP_{\action_t}$.
\end{example}

When $\spaceofparameter=\RR^K$, Example~\ref{example-MAB} reduces to the classical stochastic multi-armed bandit. In many applications, prior knowledge can be expressed through structural constraints. 
One example is the \emph{unimodal bandit} \citep{YMa11,CPr14}, in which the expected reward increases and then decreases along the arm ordering:
\begin{align}
	\spaceofparameter = \spaceofparameter_1 \cup \cdots \cup \spaceofparameter_K,
	\qquad\text{where}\qquad
	\spaceofparameter_j
	=
	\left\{
	\parametervector\in\RR^K:
	\parameter_1\le \cdots \le \parameter_j
	\text{ and }
	\parameter_j\ge \cdots \ge \parameter_K
	\right\}.
	\label{eqn-unimodal}
\end{align}
This can be generalized to unimodality with respect to a given graph structure. Another example is the \emph{Lipschitz bandit} \citep{KSU08}: for some known constant $L>0$ and metric $\dist$ on $[K]$,
\begin{align}
	\spaceofparameter
	=
	\left\{
\parametervector\in\RR^K:
	|\parameter_i-\parameter_j|
	\le L \, \dist(i,j),
	\quad \forall i,j\in[K]
	\right\}.
	\label{eqn-Lipschitz}
\end{align}
One may also impose several structural constraints simultaneously.

We next give an example in which the objective is not the expected value of the feedback.

\begin{example}[Dynamic pricing \citep{Den15}]\label{example-pricing}
	The set $\spaceofaction\subseteq(0,+\infty)$ consists of feasible prices for a product. Each $\action \in \spaceofaction$ induces a demand distribution $\distP_{\action}$ over $[0,+\infty)$. 
	The objective value $f(\action)$ is the expected revenue $\action\, \EE_{\demandrv\sim\distP_{\action}}\demandrv$. At time $t$, the feedback $\feedback_t$ is drawn from $\distP_{\action_t}$.
\end{example}

To make decisions under uncertainty, the agent must quantify and update its beliefs over time. The Bayesian paradigm provides a coherent framework for this purpose. 
Consider a family of problem instances $\{\instance_{\parameter}\}_{\parameter\in\spaceofparameter}$ indexed by a parameter $\parameter\in\spaceofparameter$. Any dataset $\dataset$ induces a likelihood function $\likelihood(\cdot;\dataset)$ over $\spaceofparameter$. In the Bayesian paradigm, $\parameter$ is drawn from a prior distribution $\distQ_0$ over $\spaceofparameter$ \citep{SSW15,Fra18}.
After observing data $\dataset_t$, the agent proceeds in two steps:
\begin{enumerate}
	\item \textbf{Belief update.} Derive the posterior distribution $\distQ_t$ from Bayes' theorem:
	\begin{align}
		\frac{\rd \distQ_t}{\rd \distQ_0}(\parameter)
		=
		\frac{
			\likelihood(\parameter;\dataset_t)
		}{
			\int_{\spaceofparameter}
			\likelihood(\parameter';\dataset_t)\,\distQ_0(\rd \parameter')
		},
		\qquad
		\parameter\in\spaceofparameter.
		\label{eqn-Bayes}
	\end{align}
	\item \textbf{Decision-making.} Choose $\action_{t+1}$ by optimizing a criterion based on $\distQ_t$ \citep{Moc74,JSW98,FPD09,Tho33,KCG12,RVa18}.
\end{enumerate}

The Bayesian framework requires a probabilistic model for the \emph{entire} problem instance. This becomes restrictive when prior knowledge is expressed through hard structural constraints rather than a tractable generative model, and posterior inference can also be computationally demanding. For a concrete illustration, consider a Lipschitz bandit defined on an undirected graph with vertex set $[K]$ and edge set $E \subseteq [K]^2$. Let the rewards be Gaussian with unit variance and unknown mean vector $\parametervector\in \RR^{K}$ satisfying the edge-wise bounds $	|\parameter_{i} - \parameter_{j}| \leq L$, $\forall  (i,j) \in E$, which are equivalent to the Lipschitz constraint \eqref{eqn-Lipschitz} with $\dist$ being the graph distance. The parameter space is a $K$-dimensional convex polyhedron with $2|E|$ linear inequalities. A full Bayesian procedure requires specifying a prior and sampling from a posterior over this set. When the graph contains cycles (e.g.,~a two-dimensional grid), exact graphical model inference is generally intractable \citep{WJo08}. General constrained sampling methods are available in principle but are computationally demanding. For instance, Markov chain Monte Carlo (MCMC) algorithms on polytopes typically require projections or proximal steps at each iteration \citep{BDM17,BEL18,HKR18}, and mixing can be slow in high dimensions.

Motivated by these challenges, we develop a minimalist Bayesian approach that places the distributional belief only on the optimizer and handles structural constraints through profile likelihood.

\section{Minimalist Thompson sampling}\label{sec-method}

We first illustrate the main idea through a canonical example.

\begin{example}[Multi-armed bandit with Gaussian rewards]\label{example-MAB-Gaussian}
	Let $K\in\ZZ_+$, $\spaceofparameter\subseteq\RR^K$, and $\sigma>0$. For $\parametervector\in\spaceofparameter$, let $\instance_{\parametervector}$ denote the multi-armed bandit in Example~\ref{example-MAB} with reward distributions $\distP_j=N(\parameter_j,\sigma^2)$ for all $j\in[K]$. The likelihood for a dataset $\dataset_t=\{(\action_i,\feedback_i)\}_{i=1}^t$ is
	\begin{align}
		\likelihood(\parametervector;\dataset_t)
		=
		\prod_{i=1}^t \frac{1}{\sqrt{2\pi}\sigma}
		\exp\!\left(
		-\frac{(\parameter_{\action_i}-\feedback_i)^2}{2\sigma^2}
		\right).
		\label{eqn-likelihood-MAB}
	\end{align}
	We place a prior distribution $\distQ_0$ on $[K]$ to encode our initial belief about which arm is optimal. The event that arm $j$ is optimal corresponds to the composite hypothesis
	\begin{align}
		H_j:\ \parametervector \in \spaceofparameter_j
		=
		\left\{
		\bv \in \spaceofparameter:\ v_j \ge v_k,\ \forall k \in [K]
		\right\}.
		\label{eqn-parameterspace-j}
	\end{align}
	
	Since the likelihood function is defined on points rather than sets, we use profile likelihood \citep{BCo94} to quantify the evidence for $H_j$. Define
	\begin{align}
		\profilelikelihood(j;\dataset_t)
		=
		\sup_{\bv \in \spaceofparameter_j}
		\likelihood(\bv;\dataset_t).
		\label{eqn-profilelikelihood}
	\end{align}
	Equivalently, this is the constrained maximum likelihood over $\spaceofparameter_j$. It also arises in the generalized likelihood ratio test for the hypothesis $H_j$. We then mimic Bayes' rule and combine $\distQ_0$ with the profile likelihood to define the generalized posterior
	\begin{align}
		\distQ_t(j)
		=
		\frac{
			\profilelikelihood(j;\dataset_t)\distQ_0(j)
		}{
			\sum_{k=1}^K \profilelikelihood(k;\dataset_t)\distQ_0(k)
		},
		\qquad j \in [K].
		\label{eqn-posterior}
	\end{align}
	Following the Thompson sampling principle, $\action_{t+1}$ may be sampled from the updated belief $\distQ_t$.
\end{example}

This construction is computationally tractable in many structured problems, as illustrated below.

\begin{remark}[Computation]\label{remark-MAB}
	Define $I_j=\{i\in[t]:\action_i=j\}$, $\hat{\mean}_j=|I_j|^{-1}\sum_{i\in I_j}\feedback_i$, and
	\[
	L(\parametervector)=\frac{1}{2\sigma^2}\sum_{j=1}^K |I_j|(\parameter_j-\hat{\mean}_j)^2.
	\]
	Then $\log \profilelikelihood(j;\dataset_t)=-\inf_{\parametervector\in\spaceofparameter_j}L(\parametervector)$ up to an additive constant. The generalized posterior becomes
	\begin{align*}
		\distQ_t(j)
		=
		\frac{
			e^{-\inf_{\parametervector\in\spaceofparameter_j}L(\parametervector)}\distQ_0(j)
		}{
			\sum_{k=1}^K
			e^{-\inf_{\parametervector\in\spaceofparameter_k}L(\parametervector)}\distQ_0(k)
		}.
	\end{align*}
When each constraint set $\spaceofparameter_j$ is convex, as in the unconstrained, unimodal, and Lipschitz bandits, these infima can be computed efficiently via convex optimization. 
\end{remark}

\begin{remark}[Reward distribution]
	The procedure extends beyond Gaussian rewards to any model with a tractable likelihood function, such as the exponential family.
\end{remark}

The construction in Example~\ref{example-MAB-Gaussian} extends to general stochastic optimization problems, yielding the MINimalist Thompson Sampling (\Alg) procedure in Algorithm~\ref{alg-MINTS}. While standard Thompson sampling begins with a prior on the full parameter and often involves high-dimensional \emph{constrained sampling}, \Alg~requires only a prior on the optimizer and profiles out the nuisance parameters through \emph{constrained optimization}. When the profile-likelihood problems are tractable, \Alg\ offers a lightweight way to incorporate structural constraints.

\begin{remark}[Reducing computational cost]\label{remark-computation}
For structured multi-armed bandits, the posterior of \Alg\ is a categorical distribution over the arms. Once it is obtained, sampling becomes a trivial operation. Hence, the only computational burden is the posterior update, which requires solving one convex optimization problem per arm. Two strategies help reduce this cost.
	\begin{itemize}
		\item \textbf{Lazy updates.} Instead of recomputing the profile likelihood after every observation, one may update only at selected times and reuse the cached posterior between updates. The schedule may have a fixed frequency or be increasingly sparse over time, amortizing the optimization cost. 
		\item \textbf{Warm-starting.} The profile likelihood problems share the same objective and the same feasible set $\spaceofparameter$; they differ only in which arm is constrained to be optimal. The solution for an arm at round $t$ is typically close to that at round $t-1$ and to that for neighboring arms. Both observations support warm-starting, which reduces the number of solver iterations substantially in practice.
	\end{itemize}
This stands in stark contrast to full Thompson sampling, whose posterior is a high-dimensional distribution over $\spaceofparameter$, often described by an unnormalized density function. A new sample of the full parameter vector is needed in each round, which requires re-running a constrained sampling algorithm. The lazy update and warm-starting strategies for \Alg\ do not directly apply.
\end{remark}

\begin{remark}[Scope and limitations]\label{remark-limitations}
\Alg\ is most suitable when the parameter space $\spaceofparameter$ is complex but the decision space $\spaceofaction$ is simple. The benefit of arm-based prior disappears when $\spaceofaction$ itself contributes to most of the problem complexity, as in combinatorial bandits \citep{CLu12}. In such settings, alternative approaches that directly exploit the structure of $\spaceofaction$ are more appropriate.
\end{remark}

\begin{algorithm}[h]
	{\bf Input:} Problem class $\{\instance_{\parameter}\}_{\parameter\in\spaceofparameter}$, likelihood function $\likelihood$, prior distribution $\distQ_0$ over decision set $\spaceofaction$.\\
	Initialize $\dataset_0=\varnothing$. Denote by $f_{\parameter}$ the objective function of instance $\instance_{\parameter}$.\\
	{\bf For $t=1,2,\dots$:}\\
	\hspace*{.6cm} Sample $\action_t$ from $\distQ_{t-1}$, receive feedback $\feedback_t$, and update the dataset to $\dataset_t=\dataset_{t-1}\cup\{(\action_t,\feedback_t)\}$.\\
	\hspace*{.6cm} Construct the profile likelihood
	\begin{align*}
		\profilelikelihood(\action;\dataset_t)
		=
		\sup
		\Big\{
		\likelihood(\parameter;\dataset_t):
		\parameter\in\spaceofparameter
		\text{ and }
		f_{\parameter}(\action)=\max_{\action'\in\spaceofaction}f_{\parameter}(\action')
		\Big\},
		\qquad \action\in\spaceofaction.
	\end{align*}
	\hspace*{.6cm} Update the generalized posterior $\distQ_t$ by
	\begin{align*}
		\frac{\rd \distQ_t}{\rd \distQ_0}(\action)
		=
		\frac{
			\profilelikelihood(\action;\dataset_t)
		}{
			\int_{\spaceofaction}
			\profilelikelihood(\action';\dataset_t)\,\distQ_0(\rd \action')
		},
		\qquad \action\in\spaceofaction.
	\end{align*}
	\caption{MINimalist Thompson Sampling (\Alg)}
	\label{alg-MINTS}
\end{algorithm}

\section{Theoretical analysis for multi-armed bandits with mean constraints}\label{sec-MAB}

We now provide theoretical guarantees for \Alg\ (\Cref{alg-MINTS}) in multi-armed bandits with mean constraints (\Cref{example-MAB}). For each arm \(j \in [K]\), let \(\mean_j\) denote its expected reward and $\Delta_j = \max_{k \in [K]} \mean_k - \mean_j$ its suboptimality gap. We measure performance using the cumulative regret
\[
\regret(T) 
= \sum_{t=1}^T \Delta_{\action_t}
= T \max_{j \in [K]} \mean_j - \sum_{t=1}^T \mean_{\action_t},
\]
and its expectation \(\EE[\regret(T)]\). The term \emph{pseudo-regret} has also been used to refer to these two quantities \citep{LSz20,ACF02,BCe12}.

This section has two goals.
{\color{black}First, we provide a non-asymptotic regret bound for \Alg\ in unstructured bandits.}
Second, we present a sharp asymptotic analysis showing that \Alg\ achieves logarithmic regret almost surely under minimal moment assumptions, and that its leading constant automatically reflects structural constraints imposed on the mean vector. Throughout the section, we study \Alg\ with the Gaussian likelihood \eqref{eqn-likelihood-MAB}. The parametric model is used only to define the algorithm; the results below do not require the reward distributions themselves to be Gaussian.

\subsection{A non-asymptotic guarantee for unstructured bandits}

{\color{black}

We conduct a non-asymptotic analysis for unstructured bandits whose reward distributions satisfy the following light tail condition.

\begin{assumption}[Sub-Gaussian reward]\label{assumption-subg}
	The reward distributions $\{ \distP_j \}_{j=1}^K$ are 1-sub-Gaussian:
	\begin{align*}
		\EE_{\reward \sim \distP_j} e^{ \lambda (\reward - \mean_j) } \leq e^{ \lambda^2 / 2 } , \qquad \forall 
		\lambda \in \RR.
	\end{align*}
\end{assumption}

Assumption \ref{assumption-subg} is standard for bandit studies \citep{LSz20}. It holds for many common distributions with sufficiently fast tail decay, including any Gaussian distribution with variance bounded by 1, or distributions supported on an interval of width 2 \citep{Hoe94}. For sub-Gaussian distributions with general variance proxies, we can reduce to this case by rescaling.

We now present a regret bound with explicit dependence on the horizon $T$, number of arms $K$, and the sub-optimality gaps of arms. The proof is deferred to \Cref{sec-cor-regret-proof}.

\begin{theorem}[Regret bound]\label{cor-regret}
	For the multi-armed bandit in \Cref{example-MAB}, run \Alg~with a uniform prior over the arms and the Gaussian likelihood \eqref{eqn-likelihood-MAB} with $\sigma > 1$. Under Assumption \ref{assumption-subg}, there exists a constant $C$ determined by $\sigma$ such that
	\begin{align*}
		& \EE[ \regret(T) ]
		\leq C 
		\bigg(
		\min \bigg\{
		\sum_{j:~ \Delta_j > 0} 
		\frac{\log T }{\Delta_j }  ,
		\sqrt{ K T \log K } 
		\bigg\}
		+ \sum_{j=1}^{K} \Delta_j
		\bigg).
	\end{align*}
\end{theorem}

This result shows the near-optimality of \Alg.\ The sum $ \sum_{j=1}^{K} \Delta_j$ stems from the fact that each arm must be pulled at least once. When the problem instance is fixed and $T$ is sufficiently large, the problem-dependent bound $	\sum_{j:~ \Delta_j > 0} \Delta_j^{-1} \log T  $ matches the lower bound for Gaussian bandits \citep{GMS19} up to a constant factor. 
For any fixed $T$, the problem-independent bound $\sqrt{TK\log K}$ matches that for Thompson sampling using Gaussian likelihood \citep{AGo17}, achieving the minimax lower bound up to a $\sqrt{\log K}$ factor \citep{BCe12}.

\Cref{cor-regret} is a corollary of the more refined result below. See \Cref{sec-thm-regret-proof} for the proof.

\begin{theorem}[Regret bound]\label{thm-regret}
	Under the setup in \Cref{cor-regret}, there exists a constant $C$ determined by $\sigma$ such that 
	\begin{align*}
		\EE [\regret(T)] \leq C 
		\inf_{ \delta \geq 0 }
		\bigg\{
		\sum_{j:~ \Delta_j > \delta}
		\bigg(
		\frac{ \log ( \max\{ T \Delta_j^2 , e \} )  }{\Delta_j } 
		+  \Delta_j  
		\bigg)
		+ T \max_{j:~\Delta_j \leq \delta} \Delta_j
		\bigg\} .
	\end{align*}
\end{theorem}
The regret bound has the same order as that in \cite{AOr10}, achieved by a carefully designed upper confidence bound algorithm with arm elimination.

}

%
%
%

While \Cref{cor-regret} already establishes the near-optimality of \Alg\ in a non-asymptotic sense, it does not reveal how structural information encoded by the feasible parameter space \(\spaceofparameter\) affects the leading constant in regret. Our next results address this question through an asymptotic analysis.

\subsection{Asymptotic analysis of structured bandits}

We now characterize the asymptotic behavior of the cumulative regret \(\regret(T)\) for \Alg\ with the Gaussian likelihood \eqref{eqn-likelihood-MAB}, \(\sigma=1\), and a uniform prior over the arms. For this analysis, we assume finite second moments of the reward distributions and a unique optimal arm.

\begin{assumption}\label{assumption-finite-variance}
	Each reward distribution \(\distP_j\) has finite variance \(\sigma_j^2 < \infty\). There exists $\opt \in [K]$ such that $\mean_{\opt} > \max_{j \neq \opt} \mean_j$.
\end{assumption}

Our first theorem gives a logarithmic upper bound on the cumulative regret that holds almost surely for any parameter space containing the true mean vector. The proof is given in \Cref{sec-thm-asymptotic-proof}.

\begin{theorem}\label{thm-asymptotic}
	Let \(\spaceofparameter\) be an arbitrary subset of \(\RR^K\) that contains the true mean vector \(\vectorofmeans\). Run \Alg\ with a uniform prior, the Gaussian likelihood \eqref{eqn-likelihood-MAB} with \(\sigma=1\), and parameter space \(\spaceofparameter\). Under Assumption \ref{assumption-finite-variance}, we have
	\[
	\limsup_{T\to\infty}
	\frac{\regret(T)}{\log T}
	\le
	\sum_{j\neq \opt} \frac{2}{\Delta_j}
	\qquad \text{a.s.}
	\]
\end{theorem}

\Cref{thm-asymptotic} shows that, regardless of the structure imposed on the feasible parameter space, the asymptotic regret of \Alg\ is never worse than the unstructured Lai--Robbins constant \citep{LRo85}. In particular, if one ignores structural information and treats the problem as an unconstrained bandit, then the classical logarithmic rate \(\sum_{j\neq \opt} 2/\Delta_j\) remains a universal upper bound. This constant is sharp for Gaussian bandits with unknown means and unit variance.

The general upper bound in \Cref{thm-asymptotic} leaves open whether \Alg\ can exploit structure to obtain a strictly smaller asymptotic constant. Our next theorem gives an affirmative answer. In the unstructured setting, \Alg\ attains the classical Lai--Robbins constant exactly. In the unimodal setting, however, the algorithm pays only for the immediate neighbors of the optimal arm, yielding a smaller constant that matches the sharp asymptotic benchmark for unimodal Gaussian bandits with unit variance \citep{CPr14}. The proof is deferred to \Cref{sec-thm-sharp-proof}.

\begin{theorem}\label{thm-sharp}
	Consider the setup of \Cref{thm-asymptotic}. If \(\spaceofparameter = \RR^K\), then
\begin{align}
	\lim_{T\to\infty}
\frac{\regret(T)}{\log T}
=
\sum_{j\neq \opt} \frac{2}{\Delta_j}
\qquad \text{a.s.}
\label{eqn-thm-sharp-1}
\end{align}
	For the unimodal bandit with \(\spaceofparameter\) given by \eqref{eqn-unimodal},
	\[
	\lim_{T\to\infty}
	\frac{\regret(T)}{\log T}
	=
	\sum_{j:\,|j-\opt|=1} \frac{2}{\Delta_j}
	\qquad \text{a.s.}
	\]
\end{theorem}

\Cref{thm-sharp} yields strong laws of large numbers for \Alg\ in canonical settings. In the unstructured case, the same results have been shown for Thompson sampling and UCB-type algorithms \citep{CKa20,FGl22}. Hence, the profiling step does not incur significant information loss in the long run. In the unimodal case, \Cref{thm-sharp} shows that \Alg\ automatically adapts to the geometry of the constraint set and achieves the optimal constant associated with local exploration around the best arm \citep{CPr14}. The algorithm behaves as if knowing which suboptimal arms are asymptotically relevant for learning.

Taken together, \Cref{thm-asymptotic,thm-sharp} suggest a general picture: the regret of \Alg\ remains logarithmic under weak distributional assumptions, and its leading constant can improve substantially when structural constraints rule out certain alternatives to the true mean vector. Our analysis concerns the almost-sure behavior of the cumulative regret \(\regret(T)\) along the sample path, rather than the more commonly studied expected regret \(\EE[\regret(T)]\). Establishing matching asymptotic results for \(\EE[\regret(T)]\) under general mean constraints remains an interesting direction for future work.

\section{Proof sketches}

We present proof sketches for our main results.

\subsection{General preparations}

We begin by introducing notation for the pull counts and empirical means.

\begin{definition}\label{defn-pulls}
	For $j\in[K]$ and $t\in\ZZ_+$, let $\pullindexset_j(t)=\{i\in[t-1]:\action_i=j\}$ and $\pullcount_j(t)=|\pullindexset_j(t)|$. When $\pullcount_j(t)\ge1$, define $	\hat\mean_j(t)=\frac{1}{\pullcount_j(t)}\sum_{i\in\pullindexset_j(t)}\feedback_i,
	$ and set $\hat\mean_j(t)=\mean_j$ when $\pullcount_j(t)=0$.
	Denote by $\tau_{j,k}$ the time of the $k$-th pull of arm $j$, let $\xi_{j,k}=k^{-1}\sum_{i=1}^k \feedback_{\tau_{j,i}}$, and let $\history_t$ be the $\sigma$-field generated by $\dataset_t$.
\end{definition}

We relate the regret to pull counts through the following decomposition:
\begin{align}
	\regret(T)
	=\sum_{t=1}^T \Delta_{\action_t}
	= \sum_{t=1}^T \sum_{j\neq \opt} \Delta_j \one ( \action_t = j )
	=\sum_{j\neq \opt}\Delta_j\,\pullcount_j(T+1).
	\label{eqn-regret-decomposition}
\end{align}
To analyze the pull counts, we represent the posterior $\distQ_t$ through the profile loss. Remark~\ref{remark-MAB} and the uniform prior yield $\distQ_t(j)= e^{-\Lambda(j,\dataset_t)} / \sum_{k=1}^K e^{-\Lambda(k,\dataset_t)} $, where
\begin{align}
	\Lambda(j,\dataset_t)=\inf_{\btheta\in\spaceofparameter_j}\bigg\{\frac12\sum_{k=1}^K \pullcount_k(t+1)\bigl[\hat\mean_k(t+1)-\parameter_k\bigr]^2\bigg\}.
	\label{eqn-Lambda}
\end{align}
Hence
\begin{align}
	\frac{	\distQ_t(j) }{ \distQ_t(i) }
	=
	e^{\Lambda(i,\dataset_t)-\Lambda(j,\dataset_t)}
	,\qquad \forall i, j \in [K].
	\label{eqn-ratio}
\end{align}

{\color{black}
\subsection{Proof sketches for \Cref{thm-regret,cor-regret}}\label{sec-thm-regret-proof-sketch}

By \eqref{eqn-regret-decomposition}, we have
\begin{align}
	\EE [\regret(T) ] 	=  \sum_{j=1}^K \Delta_j \EE [ \pullcount_j(T+1) ] 
	=  \sum_{j=1}^K \Delta_j \bigg( \sum_{t=1}^{T}  \PP ( \action_t = j ) \bigg)  .
	\label{eqn-thm-regret-1}
\end{align}
Next, we invoke a lemma on the expected number of pulls of any sub-optimal arm. The proof borrows ideas from the analysis of Thompson sampling by \cite{AGo17} and is deferred to \Cref{sec-thm-regret-proof}.

\begin{lemma}\label{lem-regret-j}
	There exists a universal constant $C_0 > 0$ such that if $\Delta_j > 0$, then
	\begin{align*}
		& 
		\sum_{t=1}^{T}  \PP ( \action_t = j ) \leq C_0
		\bigg(
		\frac{ \sigma^2}{1 - \sigma^{-2}} \cdot  \frac{ \log ( \max\{ T \Delta_j^2 , e \} ) }{ \Delta_j^2 } 
		+ \frac{1}{ \sqrt{1 - \sigma^{-2}} }
		\bigg)
		.
	\end{align*}
\end{lemma}

Choose any $\delta \geq 0$. Then,
\begin{align*}
	\sum_{j:~  \Delta_j \leq \delta} \Delta_j \bigg( \sum_{t=1}^{T}  \PP ( \action_t = j ) \bigg) 
	\leq 
	\max_{j:~\Delta_j \leq \delta} \Delta_j \cdot
	\sum_{t=1}^{T} \sum_{j = 1}^K \PP ( \action_t = j )
	= T \max_{j:~\Delta_j \leq \delta} \Delta_j.
\end{align*}
When $ \Delta_j >  \delta$, we use \Cref{lem-regret-j} to obtain that
\begin{align*}
	\Delta_j  \sum_{t=1}^{T}  \PP ( \action_t = j ) 
	& \lesssim
	\frac{   \log ( \max\{ T \Delta_j^2 , e \} )  }{\Delta_j } 
	+  \Delta_j  ,
\end{align*}
where $\lesssim$ hides a constant factor determined by $\sigma$. Combining the estimates with \eqref{eqn-thm-regret-1} yields \Cref{thm-regret}.

The obtained regret bound is expressed as an infimum over $\delta \geq 0$. We will take $\delta = 0$ and $\delta = e \sqrt{ T^{-1} K \log K }$ to show that up to a constant factor, $\EE[ \regret(T) ]$ is bounded by
\[
\sum_{j:~ \Delta_j > 0}
\frac{ \log T}{\Delta_j} + \sum_{j=1}^K \Delta_j
\qquad\text{and}\qquad
\sqrt{ T K \log K } + \sum_{j=1}^K	 \Delta_j  ,
\]
respectively. This proves \Cref{cor-regret}.

}

\subsection{Proof sketch for \Cref{thm-asymptotic}}\label{sec-thm-asymptotic-proof-sketch}

Given \eqref{eqn-regret-decomposition}, it suffices to prove for any $j\neq \opt$ that
\begin{align}
	\limsup_{T\to\infty}\frac{\pullcount_j(T)}{\log T}\le \frac{2}{\Delta_j^2}
	\qquad\text{a.s.}
	\label{eqn-pathwise-bound}
\end{align}

We will first lower bound the pull count of the optimal arm \(\opt\) by analyzing its posterior mass \(\distQ_t(\opt)\). To this end, choose \(\hat\action_t\in\argmax_{i\in[K]}\distQ_t(i)\). Since \(\distQ_t(\hat\action_t)\ge 1/K\), \eqref{eqn-ratio} and \eqref{eqn-Lambda} give
\begin{align}
	\distQ_t(\opt)
	=
	e^{\Lambda(\hat\action_t,\dataset_t)-\Lambda(\opt,\dataset_t)}\distQ_t(\hat\action_t)
	\ge
	\frac{e^{-\Lambda(\opt,\dataset_t)}}{K}
	\ge
	\frac{1}{K}
	\exp\!\left(
	-\frac12\sum_{k=1}^K \pullcount_k(t+1)[\hat\mean_k(t+1)-\mean_k]^2
	\right).
	\label{eqn-thm-asymptotic-1}
\end{align}
We will apply the Law of the Iterated Logarithm (LIL) to upper bound $\pullcount_k(t+1)[\hat\mean_k(t+1)-\mean_k]^2$ over time, which leads to a lower bound on \(\distQ_t(\opt)\).

Second, for any suboptimal arm \(j\neq \opt\), \eqref{eqn-ratio} and \(\distQ_t(\opt)\le 1\) imply
\begin{align}
	\distQ_t(j)
	=
	e^{\Lambda(\opt,\dataset_t)-\Lambda(j,\dataset_t)}\distQ_t(\opt)
	\le
	e^{\Lambda(\opt,\dataset_t)-\Lambda(j,\dataset_t)}.
	\label{eqn-thm-asymptotic-2}
\end{align}
We will show that if \(\pullcount_j(t)\) is large, then \(\Lambda(j,\dataset_t)-\Lambda(\opt,\dataset_t)\) must also be large, forcing \(\distQ_t(j)\) to be small. A conditional Borel--Cantelli argument then yields \eqref{eqn-pathwise-bound}. See \Cref{sec-thm-asymptotic-proof} for proof details.

\subsection{Proof sketch for \Cref{thm-sharp}}\label{sec-thm-sharp-proof-sketch}

We will show that the bound \eqref{eqn-pathwise-bound} is tight for any suboptimal arm whose mean can be raised to \(\mean_{\opt}\) without violating the constraints. Below is the formal statement. The proof is in \Cref{sec-lem-arm2-log-rate-proof}.

\begin{lemma}\label{lem-arm2-log-rate}
	Consider the setup of \Cref{thm-asymptotic}. Denote by \(\{ \bm{e}_k \}_{k=1}^K\) the canonical basis of \(\RR^K\). If \(j \in [K] \setminus \{ \opt \}\) satisfies \(\vectorofmeans + ( \mean_{\opt} - \mean_j )  \bm{e}_j \in \spaceofparameter\), then
	\[
	\lim_{T\to\infty} \frac{ \pullcount_j(T) }{ \log T } = \frac{2}{ \Delta_j^2 }
	\qquad\text{a.s.}
	\]
\end{lemma}

In the unstructured bandit, every suboptimal arm $j \neq \opt$ satisfies the condition of \Cref{lem-arm2-log-rate}. The resulting pull count asymptotics and the decomposition \eqref{eqn-regret-decomposition} lead to \eqref{eqn-thm-sharp-1}.
For the unimodal bandit, \Cref{lem-arm2-log-rate} applies to the immediate neighbors of \(\opt\). It suffices to show that
\begin{align}
	\lim_{T\to\infty} \frac{ \pullcount_j(T) }{ \log T } = 0
	\qquad\text{almost surely when }
	|j - \opt| \geq 2.
	\label{eqn-unimodal-0}
\end{align}
In words, suboptimal arms not adjacent to $\opt$ get eliminated quickly. See \Cref{sec-eqn-unimodal-0-proof} for the proof.

\section{Numerical experiments}\label{sec-experiments}

We evaluate \Alg\ on a multi-armed bandit with Gaussian rewards (\Cref{example-MAB-Gaussian}).\footnote{{\color{black}
The Python code for reproducing the results is available at \url{https://github.com/kw2934/MINTS}.}}
There are $K = 12$ arms, the expected rewards are given by a unimodal vector:
\[
\vectorofmeans = (0,\, 0.1,\, 0.2,\, 0.3,\, 0.4,\, \textbf{0.7},\, 0.5,\, 0.5,\, 0.4,\, 0.2,\, 0.2,\, 0.1),
\]
and all variances are equal to $\sigma = 1$. We run $100$ independent replications in a horizon of $T=20000$ time steps. We implement two groups of bandit algorithms:

\begin{itemize}
\item Three unstructured algorithms: \Alg, UCB [the UCB1 algorithm in \cite{ACF02}, which is the same as the KL-UCB algorithm in \cite{GCa11} for Gaussian rewards with the recommended choice of hyperparameter $c=0$], and Thompson sampling \citep{Tho33}.

\item Four unimodal algorithms: \Alg, OSUB \citep{CPr14}, UTS \citep{PTR17}, and OSSB \citep{CMP17}. Both \Alg\ and OSSB can handle bandits with general mean constraints, and we apply them to the unimodal constraint set in \eqref{eqn-unimodal}.
\end{itemize}

\begin{figure}[h!]
	\centering
	\includegraphics[width=\textwidth]{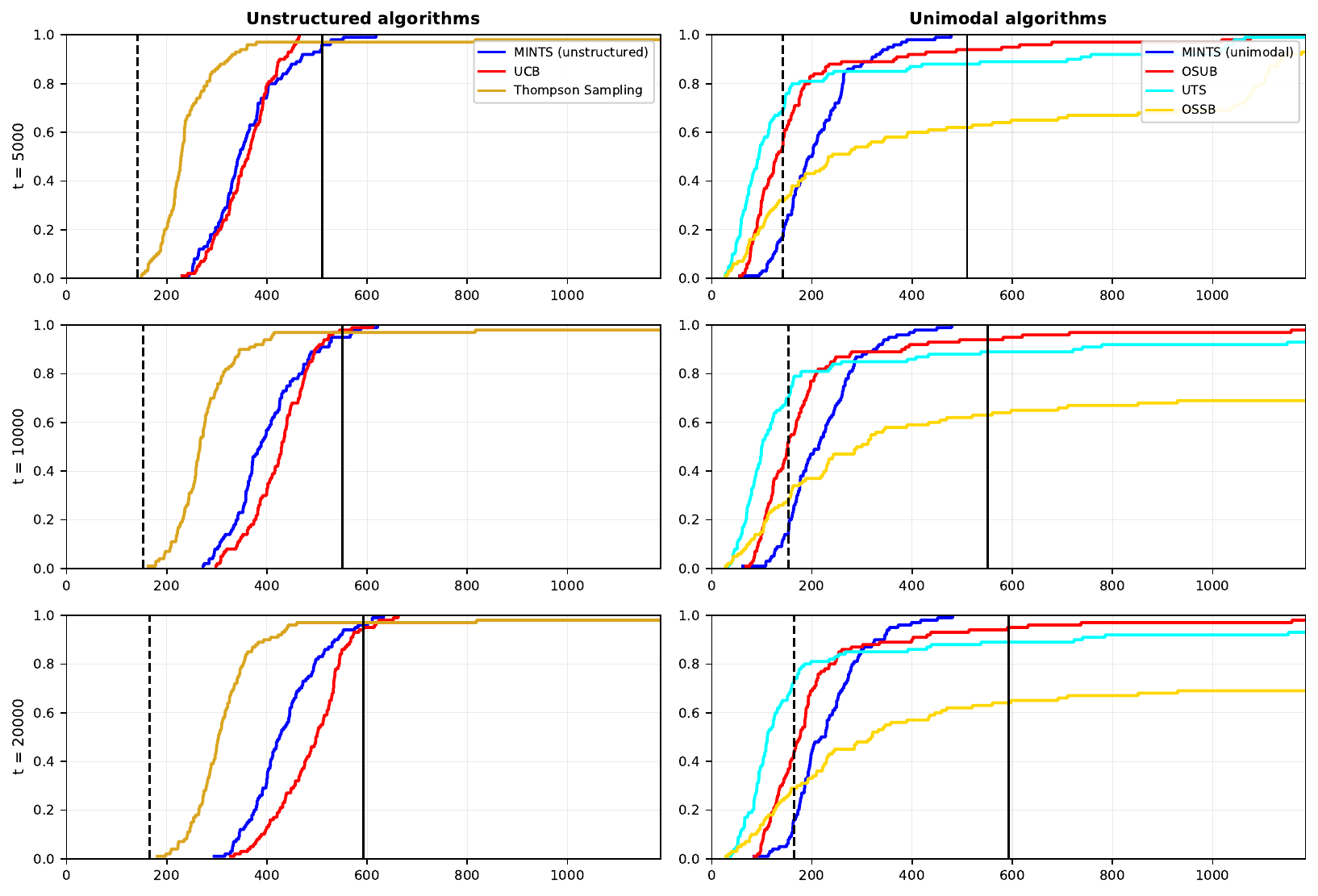}
	\caption{Empirical cumulative distribution functions of regrets at $t \in \{5000, 10000, 20000\}$ for unstructured algorithms (left column) and unimodal algorithms (right column). Vertical lines mark the asymptotic lower bounds: unstructured (solid) and unimodal (dashed).
	}
	\label{fig:hist}
\end{figure}

Figure~\ref{fig:hist} shows the empirical cumulative distribution functions of cumulative regrets at three checkpoints, $t \in \{$5000, 10000, 20000$\}$, over the 100 independent runs. The solid vertical line shows the asymptotic lower bound $ \sum_{j\neq \opt} \frac{2 \log t }{\Delta_j } $ in the unstructured setting \citep{LRo85}. Similarly, the dashed vertical line corresponds to the smaller lower bound $\sum_{j:\,|j-\opt|=1}  \frac{2 \log t }{\Delta_j } $ in the unimodal setting \citep{CPr14}.

In the unstructured setting, \Alg\ achieves slightly lower regret than UCB, and Thompson sampling performs best among the three algorithms. Their cumulative regrets all concentrate below the Lai--Robbins benchmark, showing better finite-sample performance than predicted by the asymptotic limit.
In the unimodal setting, all structured algorithms track the Combes--Proutiere benchmark. \Alg\ has the tightest concentration, whereas OSUB, UTS and OSSB have significantly heavier tails. The latter three rely on point estimates of the optimal arm and are susceptible to misidentification under high noise (our noise level $\sigma = 1$ is larger than the suboptimality gaps in $\vectorofmeans$). Once trapped, they accumulate regret for extended periods before recovering.

The numerical results support our asymptotic analysis in \Cref{thm-sharp} and demonstrate competitive performance of \Alg. While profiling out nuisance components can lead to information loss (\Alg\ underperforms Thompson sampling in the unstructured case), the impact does not seem consequential. By contrast, the flexibility gain in the structured case is substantial.

\section{Discussion}\label{sec-discussion}

We introduced a minimalist Bayesian framework for stochastic optimization that only requires a prior for the component of interest and handles nuisance parameters via profile likelihood. The lightweight modeling makes it easy to incorporate structural constraints on problem parameters, opening several promising avenues for future research. First, a precise characterization of the arm-pull dynamics of \Alg\ remains an open problem, complementing a recent line of work on UCB and Thompson sampling \citep{KZe21,FGl25,HKZ24,Han26}. Second, designing scalable algorithms for sampling from the generalized posterior is critical for handling continuous or high-dimensional spaces. Third, developing more sophisticated acquisition rules beyond simple posterior sampling could further improve performance. Beyond these refinements, extending the minimalist principle to contextual bandits and reinforcement learning presents an exciting frontier.

\section*{Acknowledgement}
The author thanks Yeon-Koo Che, Yaqi Duan and Chengpiao Huang for helpful discussions. This research is supported by NSF grants DMS-2210907 and DMS-2515679.


\appendix

\section{Proofs of \Cref{cor-regret,thm-regret}}

\subsection{Proof of \Cref{cor-regret}}\label{sec-cor-regret-proof}

The result is trivial when $K = 1$ or $T = 1$. From now on, we assume $K \geq 2$, $T \geq 2$ and use $\lesssim$ to hide constant factors determined by $\sigma$.

By taking $\delta = 0$ in the regret bound in \Cref{thm-regret}, we obtain that
\[
\EE[ \regret(T) ]
\lesssim
\sum_{j:~ \Delta_j > 0}
\frac{ \log (  
	\max \{ T \Delta_j^2, e \} 	
	) }{\Delta_j } 
+ 
\sum_{j=1}^K
\Delta_j  .
\]
Note that 
\begin{align*}
	\frac{ \log (  
		\max \{ T \Delta_j^2, e \} 	
		) }{\Delta_j } 
	& = \frac{ 
		\max \{ \log T + 2 \log \Delta_j, 1 \}
	}{\Delta_j } 
	\leq\frac{1 + \log T}{\Delta_j} + \frac{ 2 \log (1 + \Delta_j ) }{\Delta_j} \\
	&
	\overset{(\mathrm{i})}{\leq}
	\frac{1 + \log T}{\Delta_j} + 2
	\overset{(\mathrm{ii})}{\leq}
	\frac{1 + \log T}{\Delta_j} + \bigg(
	\Delta_j + \frac{1}{\Delta_j}
	\bigg)
	\lesssim \frac{ \log T}{\Delta_j} + \Delta_j
	,
\end{align*}
where $(\mathrm{i})$ and $(\mathrm{ii})$ follow from elementary inequalities $\log (1+z) \leq z$ and $z + 1/z \geq 2$ for all $z > 0$. As a result, we get
\begin{align}
	\EE[ \regret(T) ]
	\lesssim
	\sum_{j:~ \Delta_j > 0}
	\frac{ \log T}{\Delta_j} + \sum_{j=1}^K \Delta_j
	\label{eqn-cor-1}
\end{align}

On the other hand, \Cref{thm-regret} implies that
\[
\EE [\regret(T)] \lesssim
\inf_{\delta \geq 0}
\bigg\{
\sum_{j:~ \Delta_j > \delta}
\frac{ \log (  
	\max \{ T \Delta_j^2, e \} 	
	) }{\delta}
+  T \delta
\bigg\}
+ \sum_{j=1}^K	 \Delta_j  .
\]
Choose any $\delta \geq e / \sqrt{T}$. When $\Delta_j > \delta$, we have $T \Delta_j^2 \geq e^2 > e$ and
\begin{align*}
	\frac{
		\log (  
		\max \{ T \Delta_j^2, e \} 	
		)	
	}{\Delta_j } =
	\frac{ 2 \log ( \sqrt{T} \Delta_j )}{\Delta_j}=
	2 \sqrt{T} \cdot \frac{ \log ( \sqrt{T} \Delta_j ) }{
		\sqrt{T} \Delta_j 
	} 
	\leq 2 \sqrt{T} \cdot  \frac{ \log ( \sqrt{T} \delta ) }{
		\sqrt{T} \delta 
	} =  \frac{ 2 \log ( \sqrt{T} \delta ) }{
		\delta 
	} ,
\end{align*}
where the inequality follows from the facts that $\sqrt{T} \Delta_j \geq \sqrt{T} \delta \geq e$ and $z \mapsto z^{-1} \log z$ is decreasing on $[e, +\infty)$. Consequently,
\[
\EE [\regret(T)]
\lesssim
\inf_{\delta \geq e / \sqrt{T}}
\bigg\{
\frac{ K \log ( \sqrt{T} \delta) }{\delta}
+  T \delta
\bigg\}
+ \sum_{j=1}^K	 \Delta_j  .
\]
Taking $\delta = e \sqrt{ T^{-1} K \log K }$, we get $	\EE [\regret(T)]
\lesssim 	\sqrt{ T K \log K } + \sum_{j=1}^K	 \Delta_j  $. The proof is completed by combining the above estimate and \eqref{eqn-cor-1}.

\subsection{Proof of \Cref{thm-regret}}\label{sec-thm-regret-proof}

Following the sketch in \Cref{sec-thm-regret-proof-sketch}, it suffices to prove \Cref{lem-regret-j}.

\subsubsection{Preparations}

Choose any $M > 0$ and $j \in [K]$ such that $\Delta_j > 0$. Define $u_j = \mean_j + \Delta_j / 3$, $v_j = \mean_j + 2 \Delta_j / 3$, and
\begin{align*}
	& \cJ_1 =  \sum_{t=1}^T \PP [ \action_t = j, \pullcount_j(t) \leq M ]
	,\qquad \cJ_2 = \sum_{t=1}^T \PP [ \action_t = j, \hat{\mean}_j(t) \geq u_j ] , \\
	& \cJ_3 = \sum_{t=1}^T \PP [ \action_t = j,\pullcount_j(t) > M, \hat{\mean}_j(t) < u_j ] .
\end{align*}
We have a decomposition
\begin{align}
	\sum_{t=1}^{T}  \PP ( \action_t = j ) \leq \cJ_1 + \cJ_2 + \cJ_3 .
	\label{eqn-J-0}
\end{align}
By definition,
\begin{align}
& 	\cJ_1 
	= \sum_{t=1}^T \EE \Big( \one [ \action_t = j, \pullcount_j(t) \leq M ] \Big)
	= \EE \bigg( 
	\sum_{t=1}^T \one [ \action_t = j, \pullcount_j(t) \leq M ]
	\bigg) \leq M + 1
	,	\label{eqn-J-3} 
	\\ 
& \cJ_2 \leq \sum_{k = 1 }^{\infty} \PP ( \xi_{j, k} \geq u_j )
	= \sum_{k = 1 }^{\infty} 
	 \PP ( \xi_{j, k} - \mean_j \geq \Delta_j / 3 )  .
		\label{eqn-J-2-1} 
\end{align}
We invoke useful concentration bounds on the difference between the empirical average reward $\xi_{j, k}$ and the expectation $\mean_j$.

\begin{lemma}\label{lem-tail}
	Under Assumption \ref{assumption-subg}, we have 
	\begin{align*}
		& \PP  (
		\xi_{j, k} - \mean_j \geq t
		) \leq  e^{-kt^2 / 2} , \qquad \forall t \geq 0 , \\
		& \PP  (
		\xi_{j, k} - \mean_j \leq - t
		) \leq  e^{-kt^2 / 2} , \qquad \forall t \geq 0 , \\
		& \EE e^{ \lambda k ( \xi_{j, k} - \mean_j )^2 /2} \leq \frac{1}{\sqrt{1 - \lambda}}, \qquad \forall \lambda \in [0, 1).
	\end{align*}
\end{lemma}

\paragraph{Proof of \Cref{lem-tail}}
Note that $\{ \xi_{j, k} - \mean_j \}_{k=1}^{\infty}$ is a martingale difference sequence with respect to the filtration $\{ \history_{\tau_{j,k}} \}_{k=1}^{\infty}$. Theorem 2.19 in \cite{Wai19} yields the desired tail bounds on $\xi_{j, k} - \mean_j $, together with the fact that $\xi_{j, k}$ is $k^{-1}$-sub-Gaussian. The proof is then completed by applying Theorem 2.6 in \cite{Wai19}.


By \eqref{eqn-J-2-1} and \Cref{lem-tail}, $\cJ_2 
\leq  \sum_{k = 0 }^{\infty} e^{-k \Delta_j^2 / 18} = ( 1 -  e^{-\Delta_j^2 / 18} )^{-1} $. For any $z > 0$, we have $e^z \geq 1 + z$ and thus $e^{-z} \leq  (1+z)^{-1}$. Then,
\begin{align}
	\cJ_2 
	\leq \frac{1}{1 -  
		 (1+\Delta_j^2 / 18)^{-1}
}
= \frac{18}{\Delta_j^2} + 1.
	\label{eqn-J-2}
\end{align}
It remains to bound $\cJ_3$. 

\subsubsection{Bounding $\cJ_3$}

Without loss of generality, we assume $\mean_1 = \max_{k \in [K]} \mean_k$ throughout the proof. As a result, $\Delta_j = \mean_1 - \mean_j$. Let $c\in (0, 1)$ be a constant to be determined, and $M'  = (1-c) M$. We have
\begin{align*}
	& \PP [ \action_t = j, \pullcount_j(t) > M, \hat{\mean}_j(t) < u_j ]  \\
	& = \PP [ \action_t = j, \pullcount_j(t) > M, \hat{\mean}_j(t) < u_j, \pullcount_1(t) > M' ,
	\hat{\mean}_1 (t) > v_j] \\
	&\quad + \PP [ \action_t = j, \pullcount_j(t) > M, \hat{\mean}_j(t) < u_j, 
	\pullcount_1(t) > M' , \hat{\mean}_1 (t) \leq v_j  ] \\
	&\quad + \PP [ \action_t = j, \pullcount_j(t) > M, \hat{\mean}_j(t) < u_j, 
	1 \leq \pullcount_1(t) < M', \hat{\mean}_1 (t) > \hat{\mean}_j (t) ] \\
	&\quad + \PP [ \action_t = j, \pullcount_j(t) > M, \hat{\mean}_j(t) < u_j, 
	1 \leq \pullcount_1(t) < M', \hat{\mean}_1 (t) \leq \hat{\mean}_j (t) ] \\
	&\quad + \PP [ \action_t = j, \pullcount_j(t) > M, \hat{\mean}_j(t) < u_j, \pullcount_1(t) = 0 ] .
\end{align*}
Denote by $\cE_{1, t}$, $\cE_{2, t}$, $\cE_{3, t}$, $\cE_{4, t}$ and $\cE_{5, t}$ the five summands on the right-hand side. We have 
\begin{align}
	\cJ_3 \leq \sum_{t=1}^{T} ( \cE_{1, t} + \cE_{2, t} + \cE_{3, t} + \cE_{4, t} + \cE_{5, t} )
	\label{eqn-J-1-0}
\end{align}
We will control the $\cE_{j, t}$'s individually. The following fact will come in handy: for any $\history_{t-1}$-measurable event $\eventA$,
\begin{align}
\PP ( \{ \action_t = j \} \cap \eventA )
& = \EE [ \PP (  \{ \action_t = j \} \cap \eventA | \history_{t-1} ) ] 
= \EE [ \PP ( \action_t = j  | \history_{t-1} ) \cdot \one (\eventA) ] 
= \EE [ \distQ_{t-1}(j) \one (\eventA) ] .
	\label{eqn-conditioning-1}
\end{align}
We also need to characterize the generalized posterior $\distQ_{t}$. Remark \ref{remark-MAB} implies that
\[
\distQ_{t} (j) 
= \frac{ e^{- 	\Lambda ( j , \dataset_{t} ) } \distQ_0 (j) }{ 
	\sum_{k=1}^{K} e^{ - 
		\Lambda ( k , \dataset_{t} ) }  \distQ_0 (k)
} ,
\]
where
\begin{align}
	\Lambda ( j , \dataset_{t} ) = 
	\min_{ \btheta \in \spaceofparameter_j } \bigg\{
	\frac{1}{2 \sigma^2 } \sum_{k=1}^{K} 
	\pullcount_k (t+1) [ \hat\mean_k (t+1) - \parameter_k ]^2 	
	\bigg\}.
		\label{eqn-Lambda-nonasymptotic}
\end{align}
Then, we have
\begin{align}
	\frac{\distQ_{t} (j) }{
		\distQ_{t} (i) 
	}
	= e^{\Lambda(i, \dataset_{t}) - \Lambda(j, \dataset_{t})} 
		\frac{\distQ_0 (j) }{
				\distQ_0 (i) 
			}
	.
	\label{eqn-ratio-nonasymptotic}
\end{align}

We invoke some useful estimates for $\Lambda$, whose proof is deferred to \Cref{sec-lem-ratio-proof}.

\begin{lemma}\label{lem-ratio-nonasymptotic}
	Suppose that $i, j \in [K]$ and $\hat{\mean}_i(t) \geq \hat{\mean}_j(t)$.
	\begin{enumerate}
		\item\label{lem-ratio-1} 
		We have
		\[
		\Lambda (j, \dataset_{t-1})
		\geq \frac{1}{2 \sigma^2} \cdot \frac{ [ \hat\mean_{i} (t)  - \hat\mean_j (t)  ]^2 }{  1/ \pullcount_{i} (t)  + 1/ \pullcount_j  (t) } .
		\]
		
		\item\label{lem-ratio-2} If $\pullcount_j(t) \geq \pullcount_i(t)$, then $ \Lambda ( i, \dataset_{t-1}) - \Lambda (j, \dataset_{t-1}) \leq 0 $.
		
		\item\label{lem-ratio-3} If $\pullcount_j(t) < \pullcount_i(t)$, then
		\[
		\Lambda ( i, \dataset_{t-1}) - \Lambda (j, \dataset_{t-1}) 
		\geq 
		- \frac{1}{2 \sigma^2}
		\cdot
		\frac{
			(  \hat\mean_i  - \hat\mean_j )^2
		}{
			1/\pullcount_j   - 1/ \pullcount_i 
		}
		.
		\]
	\end{enumerate}
\end{lemma}

We are now in a position to tackle the summands $\{ \cE_{j, t} \}_{j=1}^5$ in \eqref{eqn-J-1-0}.

\paragraph{Bounding $\cE_{1, t}$.}
Let $\eventA$ be the event $\{  \pullcount_j(t) > M, \hat{\mean}_j(t) < u_j, \pullcount_1(t) > M' ,
\hat{\mean}_1 (t) > v_j \}$. Choose $\hat\action_t \in \argmax_{j \in [K]} \hat{\mean}_j(t)$. We have $\Lambda (
\hat\action_t
, \dataset_{t-1}) = 0$. Then, the relation \eqref{eqn-ratio-nonasymptotic} and the uniformity of $\distQ_0$ yield
\[
\distQ_{t-1}(j) 
= e^{ \Lambda (
	\hat\action_t
	, \dataset_{t-1}) - \Lambda (j, \dataset_{t-1})}
\distQ_{t-1}(\hat{\action}_t) 
\leq e^{ -\Lambda (j, \dataset_{t-1})}.
\]
Under $\eventA$, Part \ref{lem-ratio-1} of \Cref{lem-ratio-nonasymptotic} implies that
\begin{align*}
\distQ_{t-1}(j) 
& 
 \leq \exp\bigg(
- \frac{1}{2 \sigma^2}
\cdot
\frac{ (v_j - u_j)^2 }{  1/ M + 1/ M' } 
\bigg)
= \exp\bigg(
- \frac{M'
\Delta_j^2 
}{ 36 \sigma^2}
\bigg).
\end{align*}
By \eqref{eqn-conditioning-1},
\begin{align}
\cE_{1, t} 
= 
 \EE [
\distQ_{t-1}(j)  \cdot \one (\eventA)
]
\leq \exp\bigg(
- \frac{M'
	\Delta_j^2 
}{ 36 \sigma^2}
\bigg).
\label{eqn-E-1}
\end{align}

\paragraph{Bounding $\cE_{2, t}$.} 
We have
\begin{align}
	\cE_{2, t} & \leq \PP [ \hat{\mean}_1 (t) \leq v_j , \pullcount_1(t) > M' ] .
\label{eqn-E-2-t-0}
\end{align}
For large $M'$, the probability should be small because if Arm 1 is pulled many times, then its empirical average reward should be close to the population mean. We will study this via concentration of self-normalized martingale \citep{APS11}.
For any $t \in \ZZ_+$, denote by $\reward_{t, 1}$ the potential reward if Arm 1 is pulled in the $t$-th round, and $\bar\history_{t-1}$ the $\sigma$-field generated by $\dataset_{t-1}$ and $\action_{t}$. Define $\eta_0 = 0$ and $\eta_t = \reward_{t, 1} - \mean_1$ for $t \geq 1$. Then, $\eta_t$ is $\bar\history_t$-measurable, $\EE (  \eta_t | \bar\history_{t-1} ) = 0$, and Assumption \ref{assumption-subg} implies that $\EE ( e^{\lambda \eta_t} | \bar\history_{t-1} ) \leq e^{\lambda^2 / 2} $, $ \forall \lambda \in \RR$. Define
\begin{align*}
& V_t = M' + \sum_{i=0}^{t-1} \one ( \action_i = 1 ) = M' + N_1(t) , \\
& S_t = \sum_{i=0}^{t-1} \eta_i \one ( \action_i = 1 ) = N_1(t) [\hat{\mean}_1 (t) - \mean_1].
\end{align*}
Choose any $\delta > 0$. Theorem 1 in \cite{APS11} implies that
\begin{align*}
\PP \bigg(
|S_t|^2 / V_t \leq 2 \log ( \delta^{-1}
 \sqrt{
 V_t / M'
}
)
,~~\forall t \geq 1
\bigg) \geq 1 - \delta.
\end{align*}
When the above event happens, for all $t \geq 1$ we have
\begin{align*}
\frac{
\{ N_1(t) [\hat{\mean}_1 (t) - \mean_1] \}^2
}{
M' + N_1(t)
}
\leq 2 \log(1 / \delta) + \log \bigg(
\frac{M' + N_1(t)}{M'}
\bigg)
\leq 2 \log(1 / \delta) + \frac{N_1(t)}{M'},
\end{align*}
where the last inequality follows from the elementary fact $\log (1+z) \leq z$, $\forall z \geq 0$. Therefore, we have
\begin{align*}
\PP \bigg(
	\frac{
	 N_1(t) 
	}{
		M' + N_1(t)
	}[\hat{\mean}_1 (t) - \mean_1]^2
	\leq \frac{ 2 \log(1 / \delta) }{N_1(t)} + \frac{1}{M'}, ~~ \forall t \geq 1
\bigg)
\geq 1 - \delta.
\end{align*}
Denote by $\eventA_{\delta}$ the above event.

We now come back to \eqref{eqn-E-2-t-0}. On the event $\{ \hat{\mean}_1 (t) \leq v_j , \pullcount_1(t) > M' \}$, we have
\begin{align*}
	\frac{
	N_1(t) 
}{
	M' + N_1(t)
}[\hat{\mean}_1 (t) - \mean_1]^2 > \frac{\Delta_j^2}{18}
\qquad\text{and}\qquad
\frac{ 2 \log(1 / \delta) }{N_1(t)} + \frac{1}{M'}
< \frac{2 \log (1 / \delta) + 1}{M'}.
\end{align*}
Take $\delta = e^{1/2 - M' \Delta_j^2 / 36}$. We have $[ 2 \log (1 / \delta) + 1 ] / M' = \Delta_j^2 / 18$ and thus
\[
\PP [  \hat{\mean}_1 (t) \leq v_j  , \pullcount_1(t) > M' ]
\leq 
\PP \bigg(
	\frac{
	N_1(t) 
}{
	M' + N_1(t)
}[\hat{\mean}_1 (t) - \mean_1]^2 
> 
\frac{ 2 \log(1 / \delta) }{N_1(t)} + \frac{1}{M'}
\bigg)
\leq \PP ( \eventA_{\delta}^c ) \leq \delta.
\]
Then, \eqref{eqn-E-2-t-0} leads to
\begin{align}
\cE_{2, t} \leq e^{1/2 - M' \Delta_j^2 / 36}.
	\label{eqn-E-2}
\end{align}

\paragraph{Bounding $\cE_{3, t}$.} 
Let $\eventA$ be the event $\{ \pullcount_j(t) > M, \hat{\mean}_j(t) < u_j, 
1 \leq \pullcount_1(t) < M', \hat{\mean}_1 (t) > \hat{\mean}_j (t) \}$. Under $\eventA$, we apply the relation \eqref{eqn-ratio-nonasymptotic} and Part \ref{lem-ratio-2} of \Cref{lem-ratio-nonasymptotic} to Arms $1$ and $j$ (as the $i$ and $j$ therein) and then obtain $  \distQ_{t-1} (j)  \leq \distQ_{t-1} (1) $. Therefore, by \eqref{eqn-conditioning-1},
\begin{align}
& \cE_{3, t}
\leq  \EE [  \distQ_{t-1} (1)  \cdot \one (\eventA) ] 
\leq 	\EE \Big(
\one [ \action_t = 1, 1 \leq \pullcount_1(t) < M'  ] 
\Big) ,
\notag \\
& \sum_{t=1}^T
\cE_{3, t}
	\leq 	\EE \bigg(
\sum_{t=1}^T	\one [ \action_t = 1, 1 \leq \pullcount_1(t) < M'  ] 
	\bigg) 
\leq \lceil M' \rceil - 1.
\label{eqn-E-3}
\end{align}

\paragraph{Bounding $\cE_{4, t}$.} Let $\eventA$ be the event $\{ \pullcount_j(t) > M, \hat{\mean}_j(t) < u_j, 
1 \leq \pullcount_1(t) < M', \hat{\mean}_1 (t) \leq \hat{\mean}_j (t)  \}$. By \eqref{eqn-conditioning-1},
\[
\cE_{4, t}
=
 \EE 
\bigg(
\frac{\distQ_{t-1}(j)}{ \distQ_{t-1} (1) } 
 \distQ_{t-1} (1)  \cdot \one (\eventA) 
\bigg)
=
\EE \bigg(
\frac{\distQ_{t-1}(j)}{ \distQ_{t-1} (1) } 
\cdot \one ( \{ \action_t = 1  \} \cap \eventA) 
\bigg) .
\]
Under $\eventA$, we can apply the relation \eqref{eqn-ratio-nonasymptotic} and Part \ref{lem-ratio-3} of \Cref{lem-ratio-nonasymptotic} to Arms $j$ and $1$ (as the $i$ and $j$ therein). This yields
\begin{align*}
\frac{\distQ_{t-1}(j)}{ \distQ_{t-1} (1) } 
& \leq 
\exp\bigg(
 \frac{1}{2 \sigma^2}
\cdot
\frac{
	[ \hat\mean_j (t) - \hat\mean_1(t) ]^2
}{
	1/\pullcount_1 (t) - 1/ \pullcount_j (t)
}
\bigg)
 \leq 
\exp\bigg(
\frac{1}{2 \sigma^2}
\cdot
\frac{
[ \mu_1 - \hat\mean_1(t) ]^2
}{
	1/\pullcount_1(t) - 1/ \pullcount_j (t) 
}
\bigg) .
\end{align*}
Since $1 \leq \pullcount_1(t) < (1-c) M < M < \pullcount_j(t)$, we have $\pullcount_j(t) > \pullcount_1(t) / (1-c)$ and
\[
\frac{1}{\pullcount_1(t)} - \frac{1}{\pullcount_j(t)} \geq 
\frac{1}{\pullcount_1(t)} - \frac{1}{\pullcount_1(t) / (1-c)} =  \frac{c}{\pullcount_1(t)}.
\]
Hence,
\begin{align*}
&
\frac{ \distQ_{t-1} (j) }
{ 
	\distQ_{t-1} (1)
}
	\leq 
	\exp\bigg(
	\frac{ \pullcount_1(t)
		[ \mu_1 - \hat\mean_1(t) ]^2
	}{ 2 c \sigma^2}
	\bigg) .
\end{align*}
We have
\begin{align*}
\sum_{t=1}^{T}
\cE_{4, t} 
& \leq
\EE 
\bigg[
\sum_{t=1}^{T}
\exp\bigg(
\frac{ \pullcount_1(t)
	[  \hat\mean_1(t) - \mu_1 ]^2
}{2 c \sigma^2}
\bigg)
\one [ \action_t = 1, 1 \leq \pullcount_1(t) < M'  ] 
\bigg] \\
& \leq \EE 
\bigg[
\sum_{k=2}^{ \lceil M' \rceil - 1}
\exp\bigg(
\frac{ (k-1)
	[  \hat\mean_1( \tau_{1,k} ) - \mu_1 ]^2
}{2 c \sigma^2}
\bigg)
\bigg]
 = \sum_{s = 1}^{ \lceil M' \rceil - 2} \EE 
\bigg[
\exp\bigg(
\frac{ s
( \xi_{1,s}  - \mu_1 )^2
}{2 c \sigma^2}
\bigg)
\bigg] ,
\end{align*}
where $\tau_{1,k}$ is the time of the $k$-th pull of Arm $1$.

When $c \sigma^2 > 1$, we obtain from \Cref{lem-tail} that
\[
 \EE 
\bigg[
\exp\bigg(
\frac{ s
	( \xi_{1,s}  - \mu_1 )^2
}{2 c \sigma^2}
\bigg)
\bigg]
\leq \frac{1}{ \sqrt{ 1 - 1 / (c \sigma^2) } } , \qquad \forall s \in \ZZ_+.
\]
Therefore,
\begin{align}
	\sum_{t=1}^{T}
	\cE_{4, t} 
	& \leq
\frac{  \lceil M' \rceil - 2 }{ \sqrt{ 1 - 1 / (c \sigma^2) } } .
	\label{eqn-E-4}
\end{align}

\paragraph{Bounding $\cE_{5, t}$.} If $\pullcount_1(t) = 0 $, then $\Lambda ( 1 , \dataset_{t-1}) = 0$. The relation \eqref{eqn-ratio-nonasymptotic} yields $ \distQ_{t-1} (j) \leq \distQ_{t-1} (1)$. Then, by \eqref{eqn-conditioning-1},
\begin{align*}
\cE_{5, t} 
& \leq \PP [ \action_t = j,  \pullcount_1(t) = 0 ] 
= \EE [ \distQ_{t-1} (j) \cdot \one ( \pullcount_1(t) = 0 ) ] 
 \\&
 \leq \EE [ \distQ_{t-1} (1) \cdot \one ( \pullcount_1(t) = 0 ) ] 
 = \EE \Big(
\one [ \action_t = 1 ,  \pullcount_1(t) = 0 ] 
\Big).
\end{align*}
We have
\begin{align}
\sum_{t=1}^{T} \cE_{5, t}
\leq 
 \EE \bigg(
\sum_{t=1}^{T}
\one [ \action_t = 1 ,  \pullcount_1(t) = 0 ] 
\bigg)
\leq 1.
\label{eqn-E-5}
\end{align}

\paragraph{Bounding $\cJ_{3}$.} 
Below we use $\lesssim$ to hide universal constant factors. Summarizing \eqref{eqn-E-1}, \eqref{eqn-E-2}, \eqref{eqn-E-3}, \eqref{eqn-E-4} and \eqref{eqn-E-5}, we get
\begin{align*}
\cJ_3
& \leq 
T  \exp\bigg(
- \frac{M'
	\Delta_j^2 
}{ 36 \sigma^2}
\bigg)
+ 
T e^{1/2 - M' \Delta_j^2 / 36}
+ (
\lceil M' \rceil - 1
)
+ 
\frac{  \lceil M' \rceil - 2 }{ \sqrt{ 1 - 1 / (c \sigma^2) } } 
+ 1 \notag \\
&\lesssim
T   \exp\bigg(
- \frac{M'
	\Delta_j^2 
}{ 36 \sigma^2}
\bigg)
+ \frac{ (M' + 1) }{
\sqrt{ 1 - 1 / (c \sigma^2) } 
} 
=
T \exp\bigg(
- \frac{ (1-c) M
	\Delta_j^2 
}{ 36 \sigma^2}
\bigg)
+ \frac{  [ (1-c) M + 1] }{
	\sqrt{ 1 - 1 / (c \sigma^2) } 
} ,
\end{align*}
so long as $1 / \sigma^2 < c < 1$. Let $c = (1 + \sigma^{-2}) / 2$. We have $1 - c = (1 - \sigma^{-2})/2$ and $1 - 1/ c \sigma^2 = (\sigma^2 - 1) / (\sigma^2 + 1) \geq (1 - \sigma^{-2})/2$. Hence,
\begin{align}
	\cJ_3
	&  \lesssim
T
\exp\bigg(
- \frac{ (1-\sigma^{-2}) M
	\Delta_j^2 
}{ 72 \sigma^2}
\bigg)
+ \frac{
1
}{
\sqrt{ 1 - \sigma^{-2} }
}
\bigg(
\frac{1 - \sigma^{-2}}{2} M + 1
\bigg) \notag\\
&  \lesssim
T
  \exp\bigg(
- \frac{ (1-\sigma^{-2}) M
	\Delta_j^2 
}{ 72 \sigma^2}
\bigg)
+  M +
\frac{
1
}{
	\sqrt{ 1 - \sigma^{-2} }
}
.
	\label{eqn-J-1}
\end{align}

\subsubsection{Final steps}

Combining \eqref{eqn-J-0}, \eqref{eqn-J-3}, \eqref{eqn-J-2} and \eqref{eqn-J-1}, we get
\begin{align*}
 \sum_{t=1}^{T}  \PP ( \action_t = j ) 
& \lesssim
\bigg[
T  \exp\bigg(
- \frac{ (1-\sigma^{-2}) M
	\Delta_j^2 
}{ 72 \sigma^2}
\bigg)
+ M +
\frac{
1
}{
	\sqrt{ 1 - \sigma^{-2} }
}
\bigg]
+ \bigg(
 \frac{1}{\Delta_j^2} + 1
\bigg)
+ (M+1) \\
& \lesssim
T 
 \exp\bigg(
- \frac{ (1-\sigma^{-2}) M
	\Delta_j^2 
}{ 72 \sigma^2}
\bigg)
+  \frac{1}{\Delta_j^2} + M  + \frac{1}{ \sqrt{1 - \sigma^{-2}} } 
, \qquad \forall M > 0.
\end{align*}
By taking $M = \frac{72 \sigma^2}{1 - \sigma^{-2}} \cdot  \frac{ \log ( \max\{ T \Delta_j^2 , e \} }{ \Delta_j^2 } $, we get
\begin{align*}
\sum_{t=1}^{T}  \PP ( \action_t = j ) 
& \lesssim
T e^{ - \log (  \max\{ T \Delta_j^2 , e \} ) }
+ \frac{1}{\Delta_j^2}
	+ 
 \frac{ \sigma^2}{1 - \sigma^{-2}} \cdot  \frac{ \log (  \max\{ T \Delta_j^2 , e \} ) }{ \Delta_j^2 } 
  + \frac{1}{ \sqrt{1 - \sigma^{-2}} }
 \\
& =   \frac{ \sigma^2}{1 - \sigma^{-2}} \cdot  \frac{ \log ( \max\{ T \Delta_j^2 , e \} ) }{ \Delta_j^2 } 
+ \frac{1}{ \sqrt{1 - \sigma^{-2}} }
.
\end{align*}

\subsection{Proof of \Cref{lem-ratio-nonasymptotic}}\label{sec-lem-ratio-proof}

\subsubsection{Part \ref{lem-ratio-1}}

For notational simplicity, we will suppress the time index $t$ in $\pullcount_k(t)$'s and $\hat{\mean}_k(t)$'s. 
The result is trivial when $\hat\mean_i = \hat\mean_j$. Below we assume that $\hat\mean_i > \hat\mean_j$. 
By \eqref{eqn-Lambda-nonasymptotic}, we have
\begin{align}
	&	2 \sigma^2 \Lambda (j, \dataset_{t-1})
	= \min_{\btheta \in \spaceofparameter_j } 
	\bigg\{
	\sum_{k=1}^{K} 
	\pullcount_k   (  \hat\mean_k  - \theta_k )^2 
	\bigg\} \notag\\
	& \geq  \min_{\btheta \in \spaceofparameter_j } 
	\bigg\{
	\pullcount_j   (  \hat\mean_j   - \theta_j )^2 
	+ 
	\pullcount_{i}   ( \hat\mean_{i}  - \theta_{i} )^2 
	\bigg\} 
	=   \min_{ \theta_j \geq \theta_{i} } 
	\bigg\{
	\pullcount_j   ( \hat\mean_j   - \theta_j )^2 
	+ 
	\pullcount_{i}   ( \hat\mean_{i}  - \theta_{i} )^2 
	\bigg\}.
\label{eqn-proof-ratio-1}
\end{align}
Denote by $h(\theta_j, \theta_i)$ the function in the brackets. The assumption $\hat{\mean}_i  > \hat{\mean}_j $ implies that for any $\theta_j$,
\[
\min_{ \theta_i \leq \theta_j }  
h(\theta_j, \theta_i)
= h \Big( \theta_j, \min \{ \hat{\mean}_i  , \theta_j \} \Big)
= \pullcount_j   (  \hat\mean_j   - \theta_j )^2 
+ 
\pullcount_{i}   ( \hat\mean_{i}  - \theta_{j} )_+^2 
.
\]
View the above as a function of $ \theta_{j} $. It is strictly increasing on $( \hat\mean_{i}  , +\infty)$. On the complement set $(-\infty, \hat\mean_{i}  ]$, the expression simplies to $\pullcount_j   ( \hat\mean_j   - \theta_j )^2 
+ 
\pullcount_{i}   (\hat\mean_{i}  - \theta_{j} )^2 $. This function's minimizer and minimum value are $( \pullcount_j    \hat\mean_j    + \pullcount_{i}    \hat\mean_{i}   ) / (  \pullcount_j    + \pullcount_{i}  )$ and $(  \hat\mean_{i}   - \hat\mean_j   )^2 / (  1/ \pullcount_{i}   + 1/ \pullcount_j   )$.
This fact and \eqref{eqn-proof-ratio-1} lead to the desired inequality. 

\subsubsection{Part \ref{lem-ratio-2}}

Choose any $ \bar\btheta \in 
\argmin_{\btheta \in \spaceofparameter_j } \ell (\btheta, \dataset_{t-1})$.

\paragraph{Case 1: $\bar\theta_j \geq \hat\mean_i  $.}
It is easily seen that $\bar\theta_i =  \hat\mean_i$. Define $\bm{\eta} \in \RR^K$ with $\eta_j = \hat\mean_j $, $\eta_i = \bar\theta_j $, and $\eta_k = 	\bar\theta_k $ for $k \notin \{ i, j \}$. We have $\bm{\eta} \in \spaceofparameter_i$ and
\begin{align*}
	& \Lambda ( i, \dataset_{t-1}) - \Lambda (j, \dataset_{t-1}) 
	= \min_{\btheta \in \spaceofparameter_i } 
	\ell (\btheta, \dataset_{t-1})
	-
	\min_{\btheta \in \spaceofparameter_j } 
	\ell (\btheta, \dataset_{t-1})
	\leq 
	\ell ( \bm\eta, \dataset_{t-1})
	- \ell (\bar\btheta, \dataset_{t-1})
	\\
	& = 
	\frac{1}{2\sigma^2}
	\bigg(
	\pullcount_i   (\hat\mean_i  - \eta_i )^2 
	+
	\pullcount_j   (\hat\mean_j  - \eta_j )^2 
	\bigg)
	-
	\frac{1}{2\sigma^2} 
	\bigg(
	\pullcount_i   (\hat\mean_i  - \bar\theta_i )^2 
	+
	\pullcount_j   ( \hat\mean_j  - \bar\theta_j )^2 
	\bigg)
	\\
	& = 
	\frac{1}{2\sigma^2}
	\bigg(
	\pullcount_i   (\hat\mean_i  - \bar\theta_j )^2  - \pullcount_j   ( \hat\mean_j  - \bar\theta_j )^2 
	\bigg) \leq 0.
\end{align*}
The last inequality follows from $\bar\theta_j \geq \hat\mean_i \geq \hat\mean_j $ and $\pullcount_j \geq \pullcount_i$.

\paragraph{Case 2: $\bar\theta_j < \hat\mean_i  $.}
It is easily seen that $\hat\mean_j \leq \bar\theta_j = \bar\theta_i$. Define $\bm{\eta} \in \RR^K$ with $\eta_j = \hat\mean_j $, $\eta_i = \hat\mean_i$, and $\eta_k = 	\bar\theta_k $ for $k \notin \{ i, j \}$. We have $\bm{\eta} \in \spaceofparameter_i$ and
\begin{align*}
	& \Lambda ( i, \dataset_{t-1}) - \Lambda (j, \dataset_{t-1}) 
	= \min_{\btheta \in \spaceofparameter_i } 
	\ell (\btheta, \dataset_{t-1})
	-
	\min_{\btheta \in \spaceofparameter_j } 
	\ell (\btheta, \dataset_{t-1})
	\leq 
	\ell ( \bm\eta, \dataset_{t-1})
	- \ell (\bar\btheta, \dataset_{t-1})
	\\
	& = 
	\frac{1}{2\sigma^2}
	\bigg(
	\pullcount_i   (\hat\mean_i  - \eta_i )^2 
	+
	\pullcount_j   (\hat\mean_j  - \eta_j )^2 
	\bigg)
	-
	\frac{1}{2\sigma^2} 
	\bigg(
	\pullcount_i   (\hat\mean_i  - \bar\theta_i )^2 
	+
	\pullcount_j   ( \hat\mean_j  - \bar\theta_j )^2 
	\bigg)
	\\
	& = 
	0 - \frac{1}{2\sigma^2}
	\bigg(
	\pullcount_i   (\hat\mean_i  - \bar\theta_j )^2  + \pullcount_j   ( \hat\mean_j  - \bar\theta_j )^2 
	\bigg) \leq 0.
\end{align*}

\subsubsection{Part \ref{lem-ratio-3}}

Choose any $\bar\btheta \in 
\argmin_{\btheta \in \spaceofparameter_i } \ell (\btheta, \dataset_{t-1})$. We claim that $\bar\theta_i \geq \hat{\mean}_i   \geq \hat{\mean}_j   = \bar\theta_j $. To show that, define $\bm{\eta} \in \RR^K$ with $\eta_j = \hat\mean_j $, $\eta_i = \max\{ \bar\theta_i ,	\hat{\mean}_i  \}$, and $\eta_k = 	\bar\theta_k $ for $k \notin \{ i, j \}$.
We have $\bm{\eta} \in \spaceofparameter_i$ and
\begin{align*}
0 \geq 2 \sigma^2 [ \ell (\bar\btheta, \dataset_{t-1}) - \ell (\bm{\eta} , \dataset_{t-1}) ]
& = 
[\pullcount_i   (\hat\mean_i  - \bar\theta_i )^2 
+
\pullcount_j   ( \hat\mean_j  - \bar\theta_j )^2 ]
-
[\pullcount_i   (\hat\mean_i  - \eta_i )^2 
+
\pullcount_j   (\hat\mean_j  - \eta_j )^2 ]
 \\
& = \pullcount_i   (\hat\mean_i  - \bar\theta_i )_-^2 
+ \pullcount_j   ( \hat\mean_j  - \bar\theta_j )^2 
\end{align*}
The inequality forces $\bar\theta_i \geq \hat\mean_i $ and $ \bar\theta_j  = \hat\mean_j $.

We now come back to the main proof. Define $\bm{\eta} \in \RR^K$ with $\eta_j = \bar\theta_i $, $\eta_i = \hat{\mean}_i$, and $\eta_k = 	\bar\theta_k $ for $k \notin \{ i, j \}$.
We have $\bm{\eta} \in \spaceofparameter_j$ and
\begin{align*}
	& \Lambda ( i, \dataset_{t-1}) - \Lambda (j, \dataset_{t-1}) 
	= \min_{\btheta \in \spaceofparameter_i } 
	\ell (\btheta, \dataset_{t-1})
	-
	\min_{\btheta \in \spaceofparameter_j } 
	\ell (\btheta, \dataset_{t-1})
	\geq 
	\ell ( \bar\btheta, \dataset_{t-1})
	- \ell (\bm{\eta}, \dataset_{t-1})
	\\
	& = 
	\frac{1}{2\sigma^2}
	\bigg(
	\pullcount_i   (\hat\mean_i  - \bar\theta_i )^2 
	+
	\pullcount_j   ( \hat\mean_j  - \bar\theta_j )^2 
	\bigg)
	-
	\frac{1}{2\sigma^2} \bigg(
	\pullcount_i   (\hat\mean_i  - \eta_i )^2 
	+
	\pullcount_j   (\hat\mean_j  - \eta_j )^2 
	\bigg)
	\\
	& = 
	\frac{1}{2\sigma^2}
	\bigg(
	\pullcount_i   (\hat\mean_i  - \bar\theta_i )^2  - \pullcount_j   ( \hat\mean_j  - \bar\theta_i )^2 
	\bigg) 
	\geq \frac{1}{2 \sigma^2} \inf_{z \geq \hat\mean_i } \bigg\{
	\pullcount_i   (\hat\mean_i  - z )^2  - \pullcount_j   (\hat\mean_j  - z )^2 
	\bigg\} \\
	& =  \frac{1}{2\sigma^2} \inf_{z \geq 0} 
	\bigg\{
	\pullcount_i    z^2  - \pullcount_j   [ z + (\hat\mean_i  - \hat\mean_j ) ]^2 
	\bigg\}.
\end{align*}
Denote by $g(z)$ the function in the bracket. From $g'(z) / 2 = ( \pullcount_i  - \pullcount_j  ) z - \pullcount_j  ( \hat\mean_i  - \hat\mean_j  ) $ and $\pullcount_i  > \pullcount_j  $, we derive that
\[
\inf_{z \geq 0} g(z) 
= g \bigg(
\frac{\pullcount_j (  \hat\mean_i  - \hat\mean_j  ) }{\pullcount_i - \pullcount_j}
\bigg)
= - \frac{\pullcount_i \pullcount_j }{\pullcount_i - \pullcount_j} (  \hat\mean_i  - \hat\mean_j  )^2
= - \frac{
	( \hat\mean_i  - \hat\mean_j  )^2
}{
	1/\pullcount_j - 1/ \pullcount_i
}.
\]

\section{Proof of \Cref{thm-asymptotic}}\label{sec-thm-asymptotic-proof}

We will follow the outline in \Cref{sec-thm-asymptotic-proof-sketch} to prove \eqref{eqn-pathwise-bound} for any $j \neq \opt$. The next two subsections present a few useful lemmas and the final proof.

\subsection{Useful lemmas}

\begin{lemma}\label{lem-sandwich}
	Suppose that $\vectorofmeans\in\spaceofparameter$. Define
	\[
	r_t=\sum_{i=1}^K \pullcount_i(t)\bigl[\hat\mean_i(t)-\mean_i\bigr]^2,\qquad
	\bar\Lambda_t(j)=\inf_{\parametervector\in\spaceofparameter_j}\bigg\{\frac12\sum_{i=1}^K \pullcount_i(t)(\parameter_i-\mean_i)^2\bigg\}.
	\]
	Choose any $\opt\in\argmax_{j\in[K]}\mean_j$. Then for any $\varepsilon>0$ and $j\in[K]$,
	\[
	(1-\varepsilon)\bar\Lambda_t(j)-2\varepsilon^{-1}r_t
	\le
	\Lambda(j,\dataset_{t-1})-\Lambda(\opt,\dataset_{t-1})
	\le
	(1+\varepsilon)\bar\Lambda_t(j)+2\varepsilon^{-1}r_t.
	\]
\end{lemma}

\noindent{\bf Proof. }
Let $\ell_t(\parametervector)=\sum_{i=1}^K \pullcount_i(t)[\parameter_i-\hat\mean_i(t)]^2$ and $\bar\ell_t(\parametervector)=\sum_{i=1}^K \pullcount_i(t)(\parameter_i-\mean_i)^2$. Expanding the square and using $2|ab|\le \varepsilon a^2+\varepsilon^{-1}b^2$ gives, for all $\parametervector\in\RR^K$,
\[
(1-\varepsilon)\bar\ell_t(\parametervector)-(\varepsilon^{-1}-1)r_t
\le
\ell_t(\parametervector)
\le
(1+\varepsilon)\bar\ell_t(\parametervector)+(\varepsilon^{-1}+1)r_t.
\]
Taking infima over $\spaceofparameter_j$ yields
\[
(1-\varepsilon)\bar\Lambda_t(j)-(\varepsilon^{-1}-1)r_t
\le
\Lambda(j,\dataset_{t-1})
\le
(1+\varepsilon)\bar\Lambda_t(j)+(\varepsilon^{-1}+1)r_t.
\]
Since $\vectorofmeans\in\spaceofparameter_{\opt}$, we have $\bar\Lambda_t(\opt)=0$, so
\[
-(\varepsilon^{-1}-1)r_t
\le
\Lambda(\opt,\dataset_{t-1})
\le
(\varepsilon^{-1}+1)r_t.
\]
Subtracting gives the claim.
\hfill$\square$

\begin{lemma}[Iterated logarithm bounds]\label{lem-lil}
	Let Assumption \ref{assumption-finite-variance} hold and choose any $C>2$. Then, for every $j\in[K]$, with probability $1$ there exists $n_j<\infty$ such that
	\begin{align}
		k(\xi_{j,k}-\mean_j)^2\le C\sigma_j^2\log\log k,\qquad \forall k\ge n_j.
		\label{lem-lil-1}
	\end{align}
	Moreover, with probability $1$, there exists $T_{\mathrm L}<\infty$ such that
	\[
	\sum_{j=1}^K \pullcount_j(t)\bigl[\hat\mean_j(t)-\mean_j\bigr]^2
	\le
	C\sum_{j=1}^K \sigma_j^2\log\log t,\qquad \forall t\ge T_{\mathrm L}.
	\]
\end{lemma}

\noindent{\bf Proof. }
Assume first that $\sigma_j>0$ for all $j$; the case $\sigma_j=0$ is immediate. For each $j$, the sequence $\{\feedback_{\tau_{j,k}}\}_{k\ge1}$ is i.i.d. with mean $\mean_j$ and variance $\sigma_j^2$. By the Hartman--Wintner law of the iterated logarithm \citep{HWi41},
\[
\limsup_{k\to\infty}\frac{\sqrt{k}\,|\xi_{j,k}-\mean_j|}{\sqrt{2\sigma_j^2\log\log k}}=1
\qquad\text{a.s.}
\]
Hence, for any $C>2$, with probability $1$ there exists $n_j<\infty$ such that \eqref{lem-lil-1} holds.

On this event, there exists $n_j'<\infty$ such that $\pullcount_j(t)[\hat\mean_j(t)-\mean_j]^2\le C\sigma_j^2\log\log t$ for all $t\ge n_j'$. Indeed, either $\pullcount_j(s)<n_j$ for all $s$, in which case the left-hand side is eventually bounded by a constant, or else $\pullcount_j(s)\ge n_j$ for some $s$, and then for all $t\ge s$,
\[
\pullcount_j(t)\bigl[\hat\mean_j(t)-\mean_j\bigr]^2
=
\pullcount_j(t)\bigl(\xi_{j,\pullcount_j(t)}-\mean_j\bigr)^2
\le
C\sigma_j^2\log\log \pullcount_j(t)
\le
C\sigma_j^2\log\log t.
\]
Taking $T_{\mathrm L}=\max_{j\in[K]} n_j'$ proves the second claim.
\hfill$\square$

\begin{lemma}\label{lem-ratio}
	Suppose that $\mean_1=\max_{j\in[K]}\mean_j$. Then
	\[
	\bar\Lambda_t(j)\ge \frac{\Delta_j^2}{2}\cdot \frac{\pullcount_1(t)\pullcount_j(t)}{\pullcount_1(t)+\pullcount_j(t)},
	\]
	with the convention $0/0=0$.
\end{lemma}

\noindent{\bf Proof. }
Write $\pullcount_k=\pullcount_k(t)$. The claim is trivial if $\pullcount_1+\pullcount_j=0$. Otherwise,
\begin{align}
	2\bar\Lambda_t(j)
	&=
	\inf_{\parametervector\in\spaceofparameter_j}
	\sum_{i=1}^K \pullcount_i(\parameter_i-\mean_i)^2
	\ge
	\inf_{\parametervector\in\RR^K:\ \parameter_j\ge \parameter_r,\ \forall r}
	\sum_{i=1}^K \pullcount_i(\parameter_i-\mean_i)^2
	\notag\\
	&\ge
	\inf_{\parameter_j\ge \parameter_1}
	\Bigl\{
	\pullcount_1(\parameter_1-\mean_1)^2+\pullcount_j(\parameter_j-\mean_j)^2
	\Bigr\}.
	\label{eqn-lem-ratio-1}
\end{align}
For fixed $\parameter_j$, the inner minimization over $\parameter_1\le \parameter_j$ is attained at $\parameter_1=\min\{\mean_1,\parameter_j\}$. Thus \eqref{eqn-lem-ratio-1} reduces to minimizing $\pullcount_j(\mean_j-\parameter_j)^2+\pullcount_1(\mean_1-\parameter_j)_+^2$ over $\parameter_j\in\RR$. The minimizer is $(\pullcount_j\mean_j+\pullcount_1\mean_1)/(\pullcount_j+\pullcount_1)$, and the minimum value is $\Delta_j^2\pullcount_1(t)\pullcount_j(t)/[\pullcount_1(t)+\pullcount_j(t)]$. Dividing by $2$ proves the lemma.
\hfill$\square$

\begin{lemma}\label{lem-pullcount}
	Let $j\in[K]$, $\varepsilon>0$, and let $\{b(t)\}_{t=1}^\infty\subseteq[0,\infty)$ be increasing with $b(t)\to\infty$. Suppose that on an event $\cA$, there exists $S<\infty$ such that
	\[
	\PP(\action_t=j\mid\history_{t-1})\le t^{-(1+\varepsilon)}
	\qquad
	\text{whenever } t\ge S \text{ and } \pullcount_j(t)\ge b(t).
	\]
	Then there exists $\cA'\subseteq\cA$ with $\PP(\cA\setminus\cA')=0$ such that on $\cA'$,
	\[
	\limsup_{t\to\infty}\frac{\pullcount_j(t)}{b(t)}\le 1.
	\]
\end{lemma}

\noindent{\bf Proof. }
Let $\cB_t=\{\action_t=j,\ \pullcount_j(t)\ge b(t)\}$. On $\cA$, for all $t\ge S$,
\[
\PP(\cB_t\mid\history_{t-1})
=
\PP(\action_t=j\mid\history_{t-1})\one[\pullcount_j(t)\ge b(t)]
\le t^{-(1+\varepsilon)}.
\]
Since $\sum_{t=1}^\infty t^{-(1+\varepsilon)}<\infty$, the conditional Borel--Cantelli lemma \citep[Theorem 5.3.2]{Dur19} yields an event $\cA'\subseteq\cA$ with $\PP(\cA\setminus\cA')=0$ on which only finitely many $\cB_t$ occur.

We now work on the event $\cA'$ and show that $\pullcount_j(t) \leq b(t)+1$ holds whenever $t$ is sufficiently large. This immediately implies the desired result because $b(t)\to\infty$.

\begin{itemize}
	\item Case 1: $\cB_t$ never occurs. Note that $\pullcount_j(1) = 0 \leq b(1) + 1$ trivially holds. Suppose $\pullcount_j(t)\le b(t)+1$ holds for some $t \geq 1$. If $\pullcount_j(t)\le b(t)$, then $\pullcount_j(t+1)\leq 
	\pullcount_j(t) + 1
	\leq b(t) + 1 \le b(t+1)+1$; if $b(t)<\pullcount_j(t)\le b(t)+1$, then necessarily $\action_t\neq j$, so $\pullcount_j(t+1) = \pullcount_j(t) \le b(t)+1 \le b(t+1)+1$. By induction, one gets $\pullcount_j(t)\le b(t)+1$ for all $t$.
	\item Case 2: $\cB_t$ occurs at least once. Let $L$ be the time of the last occurrence. Then $\action_L=j$ and $\pullcount_j(L)\ge b(L)$, hence $\pullcount_j(L+1)=\pullcount_j(L)+1$. Let $\widetilde L$ be the smallest $t\ge L+1$ such that $\pullcount_j(L+1)\le b(t)+1$, which exists because $b(t)\to\infty$. For $t\in\{L+1,\dots,\widetilde L\}$, we have $\pullcount_j(t)>b(t)$ and, since $\cB_t$ cannot occur, $\action_t\neq j$. Consequently, $\pullcount_j(\widetilde L)=\pullcount_j(L+1)$. 
	The definition of $\widetilde L$ implies that $\pullcount_j(L+1)\le b(\widetilde{L})+1$ and thus, $\pullcount_j(\widetilde L)\leq b(\widetilde{L})+1$. The induction in Case 1 shows $\pullcount_j(t)\le b(t)+1$ for all $t\ge \widetilde L$. 
\end{itemize}
\hfill$\square$

\subsection{Proof of Inequality \eqref{eqn-pathwise-bound}}

Without loss of generality, assume that $\opt=1$. 

\noindent
{\bf Step 1: A lower bound on pulls of the optimal arm.}
Fix $C>2$ and let $A=(C/2)\sum_{i=1}^K \sigma_i^2$. 
Let $\Omega_{\mathrm L}$ be the almost-sure event on which the second conclusion of \Cref{lem-lil} holds.
On $\Omega_{\mathrm L}$, there exists $T_{\mathrm L}<\infty$ such that
\begin{align}
	\frac12\sum_{i=1}^K \pullcount_i(t)\bigl[\mean_i-\hat\mean_i(t)\bigr]^2
	\le
	A\log\log t,
	\qquad \forall t\ge T_{\mathrm L}.
	\label{eqn-lambda-one-upper-structured}
\end{align}
By \eqref{eqn-thm-asymptotic-1},
\begin{align}
	\distQ_{t-1}(1)
	\ge
	\frac{1}{K(\log t)^A},
	\qquad \forall t\ge T_{\mathrm L}.
	\label{eqn-q-one-lower-structured}
\end{align}

Next, we relate the posterior mass to the number of pulls. Define
\[
M_t = \pullcount_1(t+1) -  \sum_{s=1}^t \distQ_{s-1}(1) = \sum_{s=1}^t [ \one(\action_s=1)-
\PP ( \action_s = 1 | \history_{s-1} ) 
].
\]
Then \(\{M_t\}_{t=1}^{\infty}\) is a martingale with respect to $\{ \history_t \}_{t=1}^{\infty}$, with bounded increments \(|M_t-M_{t-1}|\le 1\). By Azuma--Hoeffding inequality \citep{Azu67},
\[
\PP\bigl(|M_t|\ge x\bigr)\le 2 e^{-x^2/(2t)}
,\qquad \forall x \geq 0.
\]
Hence, $\sum_{t=1}^{\infty} \PP\bigl(|M_t|\ge 2 \sqrt{t \log t} \bigr) \leq \sum_{t=1}^{\infty} 2 t^{-2} < \infty$. By the Borel--Cantelli lemma, with probability 1, there exists $T_{\mathrm{A}}< \infty$ such that $|M_t| \leq 2 \sqrt{t \log t}$, $\forall t \geq T_{\mathrm{A}}$. Denote by $\Omega_{\mathrm{A}}$ the above event.

Suppose that $\Omega_{\mathrm{L}} \cap \Omega_{\mathrm{A}}$ occurs. The bound \eqref{eqn-q-one-lower-structured} implies that when $t \geq 2 T_{\mathrm{L}}$, we have
\begin{align*}
	\sum_{s=1}^{t}
	\distQ_{s-1}(1)
	& \geq \sum_{s= T_{\mathrm{L}} }^{t}
	\distQ_{s-1}(1)
	\geq \frac{ t - T_{\mathrm{L}} }{
		K(\log t)^A
	}
	\geq \frac{ t }{
		2K(\log t)^A
	}.
\end{align*}
Hence, when $t \geq \max \{ 2 T_{\mathrm{L}}, T_{\mathrm{A}} \}$,
\begin{align*}
	\pullcount_1(t+1) = 
	\sum_{s=1}^{t}
	\distQ_{s-1}(1) + M_t
	\geq  \frac{ t }{
		2K(\log t)^A
	} - 2 \sqrt{t \log t}.
\end{align*}
Since $\sqrt{t\log t}=o ( t(\log t)^{-A} )$, there exists $T_1< \infty$ such that 
\begin{equation}
	\pullcount_1(t)
	\geq 
	\frac{ t }{
		3K(\log t)^A
	} 	,
	\qquad \forall t\ge T_1.
	\label{eqn-optimal-arm-polylog-lower}
\end{equation}

\medskip
\noindent
{\bf Step 2: A posterior upper bound for a suboptimal arm.}
Fix $j\in[K]\setminus\{1\}$. We show that if $\pullcount_j(t)$ is large, then $\distQ_{t-1}(j)$ must be small.

Fix $\varepsilon\in(0,1)$. Let $r_t=\sum_{i=1}^K \pullcount_i(t)[\mean_i-\hat\mean_i(t)]^2$ and $d_j=(1-\varepsilon)\Delta_j^2/2$. In the rest of the proof, we adopt the convention that $0/0=0$. By \eqref{eqn-thm-asymptotic-2}, \Cref{lem-sandwich} and \Cref{lem-ratio},
\begin{align*}
	-\log \distQ_{t-1}(j) \geq
	\Lambda(j,\dataset_{t-1}) - \Lambda(1,\dataset_{t-1}) 
	\geq (1-\varepsilon) \bar\Lambda_t(j) - 2 \varepsilon^{-1} r_t
	\geq  d_j  \cdot 
	\frac{\pullcount_1(t) \pullcount_j(t)}
	{ \pullcount_1(t)+ \pullcount_j(t)} - 2 \varepsilon^{-1} r_t.
\end{align*}
On $\Omega_{\mathrm L}$, \Cref{lem-lil} gives $r_t\le 2A\log\log t$ for all $t\ge T_{\mathrm L}$. Hence,
\begin{align}
	-\log \distQ_{t-1}(j)
	\ge
	d_j\cdot \frac{\pullcount_1(t)\pullcount_j(t)}{\pullcount_1(t)+\pullcount_j(t)}
	-4A\varepsilon^{-1}\log\log t,
	\qquad \forall t\ge T_{\mathrm L}.
	\label{eqn-posterior-upper-j}
\end{align}

Now fix $\eta\in(0,1/2)$. Since $\log\log t=o(\log t)$, there exists $m<\infty$ such that $4A\varepsilon^{-1}\log\log t<\eta\log t$ for all $t\ge m$. 
To bound $\frac{\pullcount_1(t) \pullcount_j(t)}
{ \pullcount_1(t)+ \pullcount_j(t)} $ from below, we invoke an elementary inequality:
\[
\frac{ab}{a+b}\ge \min\{\eta a,(1-\eta)b\},\qquad \forall a \geq 0,~~b \geq 0,~~\eta \in[0,1].
\]
Indeed, if \(\eta \le b/(a+b)\), then \(\eta a\le ab/(a+b)\); otherwise \(1-\eta<a/(a+b)\), so \((1-\eta)b\le ab/(a+b)\).
Thus,
\[
\frac{\pullcount_1(t)\pullcount_j(t)}{\pullcount_1(t)+\pullcount_j(t)}
\ge
\min\{\eta\pullcount_1(t),(1-\eta)\pullcount_j(t)\}.
\]
On $\Omega_{\mathrm L}\cap\Omega_{\mathrm A}$, the lower bound \eqref{eqn-optimal-arm-polylog-lower} implies that there exists $m'\ge T_{\mathrm L}$ such that $\eta\pullcount_1(t)\ge (1+2\eta)\log t/d_j$ for all $t\ge m'$. Combining these bounds with \eqref{eqn-posterior-upper-j}, we obtain
\[
-\log \distQ_{t-1}(j)\ge
\min\{(1+2\eta)\log t,\ (1-\eta)d_j\pullcount_j(t)\}-\eta\log t,
\qquad \forall t\ge \max\{m,m'\}.
\]
Define $T_j = \max \{ m, m' \}$ and $b_j(t)=
\frac{1+2\eta}{1-\eta}\cdot \frac{\log t}{d_j}$. On the event $\Omega_{\mathrm{L}} \cap \Omega_{\mathrm{A}}$, we have
\begin{equation*}
	\distQ_{t-1}(j)\le t^{-(1+\eta)}
	\qquad \text{whenever } t \geq T_j \text{ and } \pullcount_j(t)\ge b_j(t)  .
\end{equation*}

\medskip
\noindent
{\bf Step 3: A pathwise logarithmic bound for suboptimal pulls.}
Applying \Cref{lem-pullcount} with $b(t)=b_j(t)$, $\cA=\Omega_{\mathrm L}\cap\Omega_{\mathrm A}$, and $S=T_j$, we get
\[
\limsup_{T\to\infty}\frac{\pullcount_j(T)}{\log T}
\le
\frac{1+2\eta}{1-\eta}\cdot \frac{1}{d_j}
=
\frac{1+2\eta}{(1-\eta)(1-\varepsilon)}\cdot \frac{2}{\Delta_j^2}
\qquad\text{a.s.}
\]
Letting $\eta\downarrow0$ and $\varepsilon\downarrow0$ yields \eqref{eqn-pathwise-bound}.

\section{Proof of \Cref{thm-sharp}}\label{sec-thm-sharp-proof}

Following the outline in \Cref{sec-thm-sharp-proof-sketch}, we will prove \Cref{lem-arm2-log-rate} and then establish \eqref{eqn-unimodal-0} for the unimodal bandit.

\subsection{Proof of \Cref{lem-arm2-log-rate}}\label{sec-lem-arm2-log-rate-proof}

Without loss of generality, assume that \(\opt=1\) and \(j=2\). In view of \eqref{eqn-pathwise-bound}, it remains to prove
\begin{align}
	\liminf_{T\to\infty} \frac{ \pullcount_2(T) }{ \log T } \geq \frac{2}{ \Delta_2^2 }
	\qquad\text{a.s.}
	\label{eqn-arm2-log-rate}
\end{align}
We will show that if $\pullcount_2(t)$ is small, then the posterior mass on arm 2 is large. By \eqref{eqn-ratio},
\[
\PP(\action_t = 2 \mid \history_{t-1})=\distQ_{t-1}(2)=\distQ_{t-1}(1)e^{-[\Lambda(2,\dataset_{t-1})-\Lambda(1,\dataset_{t-1})]}.
\]
Define \(r_t=\sum_{i=1}^K \pullcount_i(t)[\hat\mean_i(t)-\mean_i]^2\) and choose any $\gamma \in (0,1)$. \Cref{lem-sandwich} gives
\[
\Lambda(2,\dataset_{t-1})-\Lambda(1,\dataset_{t-1})
\le
(1+\gamma)\inf_{\parametervector\in\spaceofparameter_2}
\bigg\{
\frac12\sum_{r=1}^K \pullcount_r(t)(\parameter_r-\mean_r)^2
\bigg\}
+2\gamma^{-1}r_t.
\]
By the assumption \(  (\mean_1,\mean_1,\mean_3,\dots,\mean_K) \in\spaceofparameter_2\), the infimum is at most \((\Delta_2^2/2)\pullcount_2(t)\). Hence,
\[
\PP(\action_t = 2 \mid \history_{t-1})
\geq \distQ_{t-1}(1) 
\exp\bigg(
- 
\frac{(1+\gamma) \Delta_2^2}{2} \pullcount_2(t)
- 2\gamma^{-1}r_t
\bigg).
\]

Fix any \(C>2\) and let \(A=(C/2)\sum_{i=1}^K \sigma_i^2\). Let \(\Omega_{\mathrm L}\) be the almost-sure event on which the second conclusion of \Cref{lem-lil} holds. By \eqref{eqn-lambda-one-upper-structured} and \eqref{eqn-q-one-lower-structured} in the proof of \Cref{thm-asymptotic}, on \(\Omega_{\mathrm L}\) there exists \(T_{\mathrm L}<\infty\) that is independent of $\gamma$, such that 
\[
\PP(\action_t = 2 \mid \history_{t-1})
\geq \frac{1}{K(\log t)^A}
\exp\bigg(
- 
\frac{(1+\gamma) \Delta_2^2}{2} \pullcount_2(t) 
- 4 \gamma^{-1} \log \log t
\bigg),
\qquad \forall t\ge T_{\mathrm L}.
\]

Fix \(\beta\in(0,2/\Delta_2^2)\). Choose \(\gamma\in(0,1)\) and \(\varepsilon\in(0,1/2)\) such that \((1+\gamma)\beta\Delta_2^2/2\le 1-2\varepsilon\). If $t\ge T_{\mathrm L}$ and \(\pullcount_2(t)\le \beta\log t\), we have $ (1+\gamma) \Delta_2^2 \pullcount_2(t) / 2  \leq (1-2 \varepsilon) \log t$ and thus,
\[
\PP(\action_t = 2 \mid \history_{t-1})
\geq  \frac{
	t^{-(1-2\varepsilon)}
}{K(\log t)^{(1+4/\gamma) A}}.
\]
Hence, there exists $T_2 < \infty$ such that
\begin{align}
	\PP(\action_t = 2 \mid \history_{t-1})
	\geq t^{-(1-\varepsilon)}
	\qquad\text{whenever } t\ge T_2 \text{ and } \pullcount_2(t)\le \beta\log t.
	\label{eqn-arm2-log-rate-0}
\end{align}

We invoke a technical lemma that mirrors \Cref{lem-pullcount}.

\begin{lemma}\label{lem-dyadic-threshold-lower}
	Let \(j \in [K]\), \(\varepsilon\in(0,1)\), and \(\beta>0\). Suppose that on an event \(\cA\), there exists \(S < \infty\) such that
	\[
	\PP (\action_t = j \mid \history_{t-1}) \geq t^{-(1-\varepsilon)}
	\qquad
	\text{whenever } t \geq S \text{ and } \pullcount_j(t) \leq \beta \log t.
	\]
	Then there exists \(\cA' \subseteq \cA\) with \(\PP ( \cA \setminus \cA') = 0\) such that on \(\cA'\), \(\liminf_{T\to\infty} \pullcount_j(T)/\log T \geq \beta\).
\end{lemma}

\Cref{lem-dyadic-threshold-lower} and \eqref{eqn-arm2-log-rate-0} imply that \(\liminf_{T\to\infty}\pullcount_2(T)/\log T\ge \beta\) a.s. Since this holds for every \(\beta\in(0,2/\Delta_2^2)\), we obtain \eqref{eqn-arm2-log-rate}. It remains to prove \Cref{lem-dyadic-threshold-lower}.

\noindent
{\bf Proof. }
Let \(m_0=\lceil \log_2 S\rceil+1\), \(b_m=\beta (m-1)_+\log 2\), and \(F_m=\{\pullcount_j(2^{m+1})<b_{m+1}\}\). We claim that on \(\cA\), \(\sum_{m=1}^\infty \PP(F_m\mid \history_{2^m-1})<\infty\). Then the conditional Borel--Cantelli lemma implies that \(F_m\) occurs only finitely often on some \(\cA'\subseteq\cA\) with \(\PP(\cA\setminus\cA')=0\). Hence, on $\cA'$, there exists $M < \infty$ such that \(\pullcount_j(2^m)\ge b_m\) holds for all \(m \geq M\).  When $t \geq 2^M$, we let \(m=\lfloor \log_2 t\rfloor\) and derive that
\[
\pullcount_j(t)\ge \pullcount_j(2^m)\ge b_m
=\beta (m-1)_+\log 2
\ge \beta(\log_2 t-2)\log 2,
\]
yielding the conclusion. It remains to prove the claim on summability.

Fix \(m\ge m_0\) and let \(I_m=\{2^m,\dots,2^{m+1}-1\}\). 
We will use a martingale argument to bound $\PP(F_m\mid \history_{2^m-1})$.
On \(F_m\), monotonicity gives \(\pullcount_j(t)\le \pullcount_j(2^{m+1})<b_{m+1}=\beta m\log 2=\beta\log(2^m)\le \beta\log t\) for all \(t\in I_m\). Hence,
\[
F_m\subseteq
\bigg\{
\pullcount_j(t)\le \beta\log t \text{ for all } t\in I_m,
\text{ and }
\sum_{t\in I_m}\one(\action_t=j)\le b_{m+1}
\bigg\}.
\]
Define \(Y_t=\one[\action_t=j\text{ and }\pullcount_j(t)\le \beta\log t]\) and \(q_t=\EE(Y_t\mid \history_{t-1})\). Then on \(F_m\), we have \(\sum_{t\in I_m}Y_t\le b_{m+1}\). Also, since \(\pullcount_j(t)\) is \(\history_{t-1}\)-measurable, \(q_t=\PP(\action_t=j\mid \history_{t-1})\one[\pullcount_j(t)\le \beta\log t]\ge t^{-(1-\varepsilon)}\one[\pullcount_j(t)\le \beta\log t]\), and therefore on \(F_m\),
\[
\sum_{t\in I_m} q_t
\ge
\sum_{t=2^m}^{2^{m+1}-1} t^{-(1-\varepsilon)}
\ge
2^m (2^{m+1})^{-(1-\varepsilon)}
>
2^{\varepsilon m-1}.
\]
Thus,
\begin{align}
	F_m\subseteq \bigg\{\sum_{t\in I_m}Y_t\le b_{m+1} \text{ and } \sum_{t\in I_m} q_t\ge 2^{\varepsilon m-1} \bigg\}.
	\label{eqn-lem-dyadic-threshold-lower}
\end{align}

Set \(L_{2^m-1}=1\) and \(L_t=\exp\!\left(-\sum_{s=2^m}^t Y_s +(1-e^{-1})\sum_{s=2^m}^t q_s\right)\) for \(t\ge 2^m\). Since \(Y_t\mid \history_{t-1}\sim \mathrm{Bernoulli}(q_t)\), we have \(\EE(e^{-Y_t}\mid \history_{t-1})=1-q_t(1-e^{-1})\le e^{-(1-e^{-1})q_t}\), so \(\{L_t\}_{t = 2^m-1}^{\infty}\) is a supermartingale and \(\EE(L_{2^{m+1}-1}\mid \history_{2^m-1})\le 1\). On \(F_m\), \eqref{eqn-lem-dyadic-threshold-lower} forces that \(L_{2^{m+1}-1}\ge \exp\!\left(-b_{m+1}+(1-e^{-1})2^{\varepsilon m-1}\right)\). Hence Markov's inequality gives
\[
\PP(F_m\mid \history_{2^m-1})
\le
\exp\!\left(b_{m+1}-(1-e^{-1})2^{\varepsilon m-1}\right).
\]
Since \(b_{m+1}\le \beta m\log 2\), the right-hand side is summable in \(m\), proving the claim.
\hfill$\square$

\subsection{Proof of \eqref{eqn-unimodal-0} for the unimodal bandit}\label{sec-eqn-unimodal-0-proof}

By symmetry, it suffices to consider \(j\le \opt-2\). We will fix any \(\gamma>0\) and then prove
\begin{align}
	\limsup_{T\to\infty}\frac{\pullcount_j(T)}{\log T}
	\le \frac{2\gamma}{\Delta_{\opt-1}^2}
	\qquad\text{a.s.}
	\label{eqn-unimodal-1}
\end{align}
Since \(\gamma>0\) is arbitrary, this implies \eqref{eqn-unimodal-0}.

By \eqref{eqn-thm-asymptotic-2}, $\PP(\action_t=j\mid \history_{t-1}) \leq 	e^{-[\Lambda(j,\dataset_t)-\Lambda(\opt,\dataset_t)]}$. Choose \(\varepsilon \in (0,1) \) such that \((1-\varepsilon)^2(1-\varepsilon+\gamma)\ge 1+2\varepsilon\). By \Cref{lem-sandwich}, we have
\[
\Lambda(j,\dataset_{t-1})-\Lambda(\opt,\dataset_{t-1})
\ge
(1-\varepsilon)\bar\Lambda_t(j)-2\varepsilon^{-1} r_t,
\]
where $r_t = \sum_{i=1}^K \pullcount_i(t)[\hat\mean_i(t)-\mean_i]^2$. 
Choose \(C>2\) and let \(A=(C/2)\sum_{i=1}^K \sigma_i^2\). By \Cref{lem-lil}, with probability \(1\), there exists \(T_{\mathrm{L}}<\infty\) such that $r_t \leq 2 A \log\log t$ holds for all $ t\ge T_{\mathrm{L}}$. Since $\log\log t = o(\log t)$, there exists $T_0 < \infty$ such that 
\begin{align}
	\Lambda(j,\dataset_{t-1})-\Lambda(\opt,\dataset_{t-1})
	\ge
	(1-\varepsilon)\bar\Lambda_t(j)- \varepsilon \log t,
	\qquad \forall t\ge T_{0}.
	\label{eqn-unimodal-2}
\end{align}

Next, we invoke a lemma to bound $\bar\Lambda_t(j)$ from below.
\begin{lemma}\label{lem-unimodal-profile}
	Suppose that $\vectorofmeans$ belongs to the unimodal space in \eqref{eqn-unimodal}. Define $\opt \in \argmax_{i\in[K]} \mean_i$ and $\Delta_i = \mean_{\opt} - \mean_i$. Let \(\bar\Lambda_t(\cdot)\) be given in \Cref{lem-sandwich}. Choose \(j\in[K]\) with \(|j-\opt|\ge 2\), and define \(\ell=\opt-1\) when \(j<\opt\), and \(\ell=\opt+1\) when \(j>\opt\). Then
	\[
	\bar\Lambda_t(j)
	\ge
	\frac{\Delta_{\ell}^2}{2}\cdot
	\frac{\bigl[\pullcount_j(t)+\pullcount_{\ell}(t)\bigr]\pullcount_{\opt}(t)}
	{\pullcount_j(t)+\pullcount_{\ell}(t)+\pullcount_{\opt}(t)}.
	\]
\end{lemma}

Recall that $j\le \opt-2$. Then, \Cref{lem-unimodal-profile} gives
\[
\bar\Lambda_t(j)
\ge
\frac{\Delta_{\opt-1}^2}{2}\,
\bigl[\pullcount_j(t)+\pullcount_{\opt-1}(t)\bigr]
\frac{\pullcount_{\opt}(t)}
{\pullcount_j(t)+\pullcount_{\opt-1}(t)+\pullcount_{\opt}(t)}.
\]
The estimates \eqref{eqn-pathwise-bound} and \eqref{eqn-optimal-arm-polylog-lower} imply that \(\pullcount_{\opt}(t)/[\pullcount_j(t)+\pullcount_{\opt-1}(t)+\pullcount_{\opt}(t)]\to 1\) a.s. Also, applying \Cref{lem-arm2-log-rate} to arm $(\opt-1)$ gives \(\pullcount_{\opt-1}(t)/\log t\to 2/\Delta_{\opt-1}^2\) a.s. Hence, with probability \(1\), there exists \(T_1<\infty\) such that for all \(t\ge T_1\),
\[
\frac{\pullcount_{\opt}(t)}
{\pullcount_j(t)+\pullcount_{\opt-1}(t)+\pullcount_{\opt}(t)}
\ge 1-\varepsilon
\qquad\text{and}\qquad
\pullcount_{\opt-1}(t)\ge \frac{2(1-\varepsilon)\log t}{\Delta_{\opt-1}^2}.
\]
On this event,
\[
\bar\Lambda_t(j)
\ge
\frac{\Delta_{\opt-1}^2}{2}
\bigg(
\pullcount_j(t)+\frac{2(1-\varepsilon)\log t}{\Delta_{\opt-1}^2}
\bigg)(1-\varepsilon),
\qquad \forall t\ge T_1.
\]
Therefore, whenever $t\ge T_1$ and \(\pullcount_j(t)\ge 2\gamma \Delta_{\opt-1}^{-2}\log t\), we have \(\bar\Lambda_t(j)\ge (1-\varepsilon)(1-\varepsilon+\gamma)\log t\). 

Combining this with \eqref{eqn-unimodal-2} and the assumption \((1-\varepsilon)^2(1-\varepsilon+\gamma)\ge 1+2\varepsilon\), we obtain that
\[
\Lambda(j,\dataset_{t-1})-\Lambda(\opt,\dataset_{t-1})
\ge
(1+\varepsilon)\log t
\qquad  \text{ whenever } t \geq \max\{ T_0, T_1\}
\text{ and }
\pullcount_j(t)\ge \frac{ 2\gamma }{ \Delta_{\opt-1}^{2} }\log t.
\]
The above lower bound on $\Lambda(j,\dataset_{t-1})-\Lambda(\opt,\dataset_{t-1})$ implies  \(\PP(\action_t=j\mid \history_{t-1})\le t^{-(1+\varepsilon)}\). Applying \Cref{lem-pullcount} with \(b(t)=2\gamma \Delta_{\opt-1}^{-2}\log t\) proves \eqref{eqn-unimodal-1}. It remains to prove \Cref{lem-unimodal-profile}.

\noindent
{\bf Proof. } By symmetry, it suffices to prove the result for \(j\le \opt-2\). Write \(N_r=\pullcount_r(t)\). If \(\parametervector\in\spaceofparameter_j\), then \(\theta_j\ge \theta_{j+1}\ge \cdots \ge \theta_{\opt-1}\ge \theta_{\opt}\). Therefore
\[
2\bar\Lambda_t(j)\ge
\inf_{a\ge b\ge c}
\Big\{
N_j(a-\mean_j)^2+N_{\opt-1}(b-\mean_{\opt-1})^2+N_{\opt}(c-\mean_{\opt})^2
\Big\}.
\]
For fixed \((b,c)\), the minimizer over \(a\ge b\) is \(a=\max\{b,\mean_j\}\). Hence
\[
2\bar\Lambda_t(j)
\ge
\inf_{b\ge c}
\Big\{
N_j(b-\mean_j)_+^2+N_{\opt-1}(b-\mean_{\opt-1})^2+N_{\opt}(c-\mean_{\opt})^2
\Big\}.
\]
Define $g(b) = N_j(b-\mean_j)_+^2+N_{\opt -1}(b-\mean_{\opt -1})^2$. Since \(\mean_j\le \mean_{\opt -1}\), the function \(g\) is nonincreasing on
\((-\infty,\mean_{\opt -1}]\) and nondecreasing on \([\mean_{\opt -1},\infty)\).
Therefore, the minimizer of \(g\) over \([c,\infty)\) is $b_c=\max\{c,\mean_{\opt -1}\}$. Since $(b_c-\mean_j)_+\ge (b_c-\mean_{\opt -1})_+=(c-\mean_{\opt -1})_+$, we obtain
\[
2\bar\Lambda_t(j)
\ge
\inf_{c\in\RR}
\Big\{
(N_j+N_{\opt -1})(c-\mean_{\opt -1})_+^2
+
N_{\opt }(c-\mean_{\opt })^2
\Big\}.
\]
The right-hand side is the same one-sided quadratic minimization problem as in the proof of \Cref{lem-ratio}, with effective count \(N_j+N_{\opt -1}\) and gap \(\Delta_{\opt -1}\). Its minimum value is $\Delta_{\opt -1}^2 
\frac{(N_j+N_{\opt -1})N_{\opt }}
{N_j+N_{\opt -1}+N_{\opt }}
$. Dividing by \(2\) gives the desired result.
\hfill$\square$

{
\bibliographystyle{ims}
\bibliography{bib}
}

\end{document}